\documentclass[11pt, letterpaper]{article}
\usepackage[utf8]{inputenc}
\usepackage{microtype}
\usepackage{graphicx}
\usepackage{booktabs}
\usepackage{hyperref}
\usepackage{amsmath}
\usepackage{amssymb}
\usepackage{amsthm}
\usepackage{amsfonts}
\usepackage{color}
\usepackage{authblk}
\usepackage[top=1in, bottom=1in, left=1in, right=1in]{geometry}
\usepackage{cite}
\usepackage{url}
\newtheorem{prop}{Proposition}

\newtheorem{corollary}{Corollary}
\newtheorem{lemma}{Lemma}
\newtheorem{definition}{Definition}

\newtheorem{theorem}{Theorem}
\newcommand{\N}{\mathcal{N}}
\makeatletter
\newcommand\footnoteref[1]{\protected@xdef\@thefnmark{\ref{#1}}\@footnotemark}
\makeatother
\usepackage{caption}
\usepackage{subcaption}
\usepackage{tikz}
\usetikzlibrary{shapes,arrows}
\usetikzlibrary{positioning}
\tikzstyle{block} = [draw, fill=white, rectangle, 
    minimum height=3em, minimum width=3em]
\tikzstyle{circ} = [draw, fill=white, circle, minimum size=2.5em]
\tikzstyle{input} = [coordinate]
\tikzstyle{output} = [coordinate]
\tikzstyle{pinstyle} = [pin edge={to-,thin,black}]

\title{Provably Correct Learning Algorithms in the Presence of Time-Varying Features Using a Variational Perspective}
\author[1]{Joseph E. Gaudio}
\author[2]{Travis E. Gibson}
\author[1]{Anuradha M. Annaswamy}
\author[3]{Michael A. Bolender}
\affil[1]{Massachusetts Institute of Technology}
\affil[2]{Brigham and Women’s Hospital and Harvard Medical School}
\affil[3]{Air Force Research Laboratory}
\date{\today}

\begin{document}

\maketitle

\begin{abstract}
Features in machine learning problems are often time-varying and may be related to outputs in an algebraic or dynamical manner. The dynamic nature of these machine learning problems renders current higher order accelerated gradient descent methods unstable or weakens their convergence guarantees. Inspired by methods employed in adaptive control, this paper proposes new algorithms for the case when time-varying features are present, and demonstrates provable performance guarantees. In particular, we develop a unified variational perspective within a continuous time algorithm. This variational perspective includes higher order learning concepts and normalization, both of which stem from adaptive control, and allows stability to be established for dynamical machine learning problems where time-varying features are present. These higher order algorithms are also examined for provably correct learning in adaptive control and identification. Simulations are provided to verify the theoretical results.
\end{abstract}

\section{Introduction}
\label{s:Introduction}

As a field, machine learning has focused on both the processes by which computer systems automatically improve through experience, and on the underlying principles that govern learning systems \cite{Duda_2001,Bishop_2006,Jordan_2015,Goodfellow-et-al-2016}. A particularly useful approach for accomplishing this process of automatic improvement is to embody learning in the form of approximating a desired function and to employ optimization theory to reduce an approximation error as more data is observed. The field of adaptive control, on the other hand, has focused on the process of controlling engineering systems in order to accomplish regulation and tracking of critical variables of interest (e.g. speed in automotive systems, position and force in robotics, Mach number and altitude in aerospace systems, frequency and voltage in power systems) in the presence of uncertainties in the underlying system models, changes in the environment, and unforeseen variations in the overall infrastructure \cite{Sastry_1989,Astrom_1995,Ioannou1996,Narendra2005}. The approach used for accomplishing such regulation and tracking is to learn the underlying parameters through an online estimation algorithm. Stability theory is employed for enabling guarantees for the safe evolution of the critical variables, and convergence of the regulation and tracking errors to zero. In both machine learning and adaptive control the core algorithm is often inspired by gradient descent or gradient flow \cite{Narendra2005}. As the scope of problems in both fields increases, the associated complexity and challenges increase as well, necessitating a better understanding of how the underlying algorithms can be designed to enhance learning and stability for dynamical, real-time learning problems.

Modifications to standard gradient descent have been actively researched within the optimization community since computing began. The seminal accelerated gradient method proposed by \cite{Nesterov_1983} has not only received significant attention in the optimization community \cite{Nesterov_2004,Beck_2009,Bubeck_2015,Carmon_2018}, but also in the neural network learning community \cite{Krizhevsky_2012,Sutskever_2013}. Nesterov's original method, or a variant \cite{Duchi_2011,Kingma_2017,Wilson_2017} are the standard methods for training deep neural networks. To gain insight into Nesterov's method, which is a difference equation, reference \cite{Su_2016} identified the second order ordinary differential equation (ODE) at the limit of zero step size. Still pushing further in the continuous time analysis of these higher order methods, several recent results have leveraged a variational approach showing that, at least in continuous time, there exists a broad class of higher order methods where one can obtain an arbitrarily fast convergence rate \cite{Wibisono_2016,Wilson_2016}. Converting back to discrete time to obtain an implementable algorithm with rates matching that of the differential equation is also an active area of research \cite{Betancourt_2018,Wilson_2018}. It should be noted that in all the aforementioned work, while the parameter update algorithm is time-varying, the features and output of the cost function are static.

The adaptive control community has also analyzed several modifications to gradient descent over the past 40 years. These modifications have been introduced to ensure provably safe learning in the presence of both structured parametric uncertainty and unstructured uncertainty due to unmodeled dynamics, magnitude saturation, delays, and disturbances \cite{Ioannou_1984,Karason_1994,Bekiaris_Liberis_2010,Gibson2013}. A majority of these modifications are applied to first order gradient-like algorithms. One notable exception is the ``high-order tuner'' proposed by \cite{Morse_1992} which has been useful in providing stable algorithms for time-delay systems \cite{Evesque_2003}.

In this paper, we consider a general class of learning problems where features (regressors) are time-varying. In comparison, most higher order methods in the machine learning literature are analyzed for the case where features are assumed to be constant \cite{Wibisono_2016}. References \cite{Hopfield_1982,Hopfield_1984,Jordan_1986,Hochreiter_1997,Dietterich_2002,Jordan_2015,Kuznetsov_2015,Hall_2015,Zinkevich_2003,Hazan_2007,Hazan_2008,Hazan_2016,Shalev_Shwartz_2011,Raginsky_2010} emphasize the need for tools for problems where either the input features are time-varying, for time-series prediction, for recurrent networks with time-varying inputs, for sequential performance, or for online optimization. Capability to explicitly handle time-varying features sequentially processed online is essential as machine learning algorithms begin to be used in real-time safety critical applications. Utilizing a common variational perspective, inspired by \cite{Wibisono_2016}, this paper will aim to realize two objectives. The first objective is the derivation of a provably correct higher order learning algorithm for  regression problems with time-varying features. The second objective of this paper is to derive a provably correct higher order online learning algorithm for uncertain dynamical systems, as often occur in adaptive identification and control problems \cite{Narendra2005}. Both objectives are realized in this paper using the notion of a ``higher order tuner'', first introduced in adaptive control in \cite{Morse_1992} and formally analyzed within the context of time delay systems in \cite{Evesque_2003}. With the variational perspective in \cite{Wibisono_2016}, it will be shown that the high-order learning will lead to a provably correct algorithm when time-varying features are present in a machine learning problem, and that it leads to stable learning for a general class of dynamical systems, going beyond the specific problems considered in \cite{Morse_1992,Evesque_2003}.

We begin with a review of time-varying features to demonstrate how they may be related to outputs in an algebraic manner as well as through the states of a dynamical system. We then propose a new class of online algorithms inspired by the high order tuners of \cite{Morse_1992,Evesque_2003}, that take into account the time variation of the features, and provide \emph{guarantees of stability and convergence with a constant regret bound}. We propose the same high order tuners for adaptive control for convergence of model tracking errors. The derivation of these algorithms comes from a unified variational approach which relates the potential, kinetic, and damping characteristics of the algorithm, and allows for continuous time variation of the features. This paper is concluded with numerical experiments demonstrating the efficacy of the derived algorithms for time-varying regression as well as adaptive control of uncertain dynamical systems. Finally we note that while we focus on models that are linear in the parameters (for ease of exposition and clarity of presentation), nonlinearly parameterized models can also be analyzed using similar Lyapunov stability approaches \cite{Annaswamy_1996,Yu_1996,Yu_1998,Loh_1999,Cao_2003}.

\section{Warmup: time-varying features and model reference adaptive control}
\label{s:Warmup}

\subsection{Time-varying regression}
\label{ss:TV_Reg}

A time-varying regression system may be expressed as $y(t)=\theta^{*T}\phi(t)$, where $\theta^*,\phi(t)\in\mathbb{R}^N$ represent the unknown constant parameter and the known time-varying feature respectively. The variable $y(t)\in\mathbb{R}$ represents the known time-varying output. Given that $\theta^*$ is unknown, we formulate an estimator $\hat{y}(t)=\theta^T(t)\phi(t)$, where $\hat{y}(t)\in\mathbb{R}$ is the estimated output and the unknown parameter is estimated as $\theta(t)\in\mathbb{R}^N$.
Define the error between the actual output and the estimated output as
\begin{equation}\label{e:error1}
    e_y(t)=\hat{y}(t)-y(t)=\tilde{\theta}^T(t)\phi(t)
\end{equation}
where $\tilde{\theta}(t)=\theta(t)-\theta^*$ is the parameter estimation error. An overview of the time-varying regression error model may be seen in Figure \ref{f:Block_Diagram_Error_1_and_2}. The differential equation for the output error is of the form
\begin{equation}\label{e:error_model1_error}
\dot{e}_y(t)=\dot{\theta}^T(t)\phi(t)+\tilde{\theta}^T(t)\dot{\phi}(t).
\end{equation}
where the time variation of the feature $\dot{\phi}(t)$ can be seen to appear. The goal is to design a rule to adjust the parameter estimate $\theta(t)$ in a continuous manner using knowledge of $\phi(t)$ and $e_y(t)$ such that $e_y(t)$ converges towards zero. A continuous, gradient descent-like algorithm is desired as the output of the regression system $y(t)$ may be corrupted by noise and feature dimensions may be large. To do so, consider the squared loss cost function: $L=\frac{1}{2}e_y^2(t)$. The gradient of this function with respect to the parameters can be expressed as: $\nabla_{\theta}L=\phi(t)e_y(t)$. The standard gradient flow algorithm (the continuous time limit of gradient descent) may be expressed as follows with user-designed gain parameter $\gamma>0$ \cite{Narendra2005}:
\begin{equation}\label{e:update_GF}
\dot{\theta}(t)=-\gamma\nabla_{\theta}L=-\gamma \phi(t)e_y(t).
\end{equation}
The parameter error model may then be stated as $\dot{\tilde{\theta}}(t)=-\gamma\phi(t)\phi^T(t)\tilde{\theta}(t)$. Stability analysis of the algorithm in (\ref{e:update_GF}) for the error model in (\ref{e:error1}) is provided in Appendix \ref{ss:Stability_TV_Reg}.

\subsection{Model reference adaptive control and identification}
\label{ss:MRAC}

In the previous subsection, the output was an algebraic combination of the elements of the feature. In a class of problems (including adaptive identification and adaptive control) the features may be related to the errors of a dynamical system. To demonstrate this, features $\phi(t)\in\mathbb{R}^N$ may be related to a measurable state $x(t)\in\mathbb{R}^{n}$ through a dynamical system with unknown constant parameter $\theta^*\in\mathbb{R}^N$ as $\dot{x}(t)=Ax(t)+b(u(t)+\theta^{*T}\phi(t))$, where $u(t)\in\mathbb{R}$ is an input to the system.\footnote{In adaptive control the input is usually designed as $u(t)=-\theta^T(t)\phi(x(t))$.} A single input system is considered here for notational simplicity, where $A\in\mathbb{R}^{n\times n}$, and $b\in\mathbb{R}^{n\times1}$ are known stable dynamics and input matrices respectively. It can be noted that the results of this paper extend naturally to multiple input systems. Additionally, it should be noted that it is common in adaptive control for the feature to be a function of the state, i.e., $\phi(t)=\phi(x(t))$. This dynamical system is akin to a linearized recurrent neural network and is similar to the dynamical systems considered in \cite{Hazan_2017,Hazan_2018a,Recht_2018,Dean_2018,Dean_2018a,Dean2018}. Similar to the linear regression case where an output estimator was created with the same form as the time-varying system, but with an estimate of the unknown parameter, a state estimator with state $\hat{x}(t)\in\mathbb{R}^{n}$ may be designed for this system as $\dot{\hat{x}}(t)=A\hat{x}(t)+b(u(t)+\theta^T(t)\phi(t))$. Define the error between the considered dynamical system and estimator dynamical system as $e(t)=\hat{x}(t)-x(t)$. The error model for identification and control schemes may then be stated as
\begin{equation}
\label{e:error_model_2}
\dot{e}(t)=Ae(t)+b\tilde{\theta}^T(t)\phi(t)
\end{equation}
where the relation of the feature to the error can be seen to be through a differential equation, which is fundamentally different from (\ref{e:error1}). An overview of the dynamical error model may be seen in Figure \ref{f:Block_Diagram_Error_1_and_2}.
\begin{figure}[t]
    \centering
\begin{subfigure}[t]{0.4\textwidth}
	\centering
    \begin{tikzpicture}[auto, node distance=1.5cm,>=latex']
    
    \node [input, name=input1] {};
    \node [circ, right of=input1, very thick] (circle1) {};
    \node [output, name=output1, right of=circle1] {};
    \draw [->, very thick] (input1) -- node [name=phi1] {$\phi$} (circle1);
    \draw [->, very thick] (circle1) -- node [name=e_y1] {$e_y$} (output1);
    \node [input, name=arrowleft1, below left=0.25cm of circle1] {};
    \node [input, name=arrowright1, above right=0.25cm of circle1] {};
    \draw [->, thick] (arrowleft1) -- (arrowright1);
    \node[fill=white] at (circle1) {$\tilde{\theta}$};
    
    \end{tikzpicture}
    \label{sf:error1}
\end{subfigure}
    ~
\begin{subfigure}[t]{0.4\textwidth}
	\centering
    \begin{tikzpicture}[auto, node distance=1.5cm,>=latex']
    
    \node [input, name=input2, right of=output1, node distance=0.75cm] {};
    \node [circ, right of=input2, very thick] (circle2) {};
    \node [block, right of=circle2, very thick] (plant) {$W(s)$};
    \node [output, name=output2, right of=plant] {};
    \draw [->, very thick] (input2) -- node [name=phi2] {$\phi$} (circle2);
    \draw [->, very thick] (circle2) -- (plant);
    \draw [->, very thick] (plant) -- node [name=e_y_2] {$e$} (output2);
    \node [input, name=arrowleft2, below left=0.25cm of circle2] {};
    \node [input, name=arrowright2, above right=0.25cm of circle2] {};
    \draw [->, thick] (arrowleft2) -- (arrowright2);
    \node[fill=white] at (circle2) {$\tilde{\theta}$};
    
    \end{tikzpicture}
    \label{sf:error2}
\end{subfigure}
	\caption{Error models. Left: Regression (\ref{e:error1}). Right: Adaptive control (\ref{e:error_model_2}), $W(s):=(sI-A)^{-1}b$.}
	\label{f:Block_Diagram_Error_1_and_2}
\end{figure}
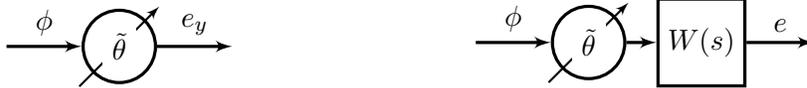
Rather than employing a gradient flow based rule, a stability based algorithm may be chosen as follows, with a gain $\gamma>0$ selected to adjust the learning rate \cite{Narendra2005}:
\begin{equation}\label{e:MRAC}
    \dot{\theta}(t)=-\gamma\phi(t) e^T(t)Pb
\end{equation}
where $P=P^T\in\mathbb{R}^{n\times n}$ is a positive definite matrix that solves the equation $A^TP+PA=-Q$, where $Q=Q^T\in\mathbb{R}^{n\times n}$ is a user selected positive definite matrix (see Appendix \ref{ss:Stability_Lyapunov}).\footnote{Open loop unstable plants may also be considered in the model tracking problem for a controllable system by choosing $\dot{\hat{x}}(t)=A_m\hat{x}(t)+b(u(t)+\theta^T(t)\phi(t))$, with $A_m\triangleq A-bK$ chosen stable with $K\in\mathbb{R}^{1\times n}$ \cite{Narendra2005}.} Comparing (\ref{e:MRAC}) to (\ref{e:update_GF}), it can be noticed that the structure is similar with the multiplication of the feature by the error. The difference between them is through the inclusion of elements from the differential equation relating the parameter error to the model tracking error (\ref{e:error_model_2}). Stability analysis of the update in (\ref{e:MRAC}) for the error model in (\ref{e:error_model_2}) is provided in Appendix \ref{ss:Stability_MRAC}.

\section{Algorithm derivation from a variational perspective}
\label{s:Derivation}

This section derives higher order update algorithms for both the time-varying regression, as well as the adaptive control and identification problems. For the time-varying regression problem, the goal is to derive a higher order algorithm to adjust the parameter estimate $\theta(t)$ (as compared to (\ref{e:update_GF})), for minimization of the estimation error $e_y(t)$ in the algebraic error model (\ref{e:error1}). For the adaptive control and identification problem, the goal is to derive a higher order algorithm to adjust $\theta(t)$ (as compared to (\ref{e:MRAC})), such that the error $e(t)$ in the dynamical error model in (\ref{e:error_model_2}) converges to zero. For the remainder of the paper, the notation of time dependence of variables will be omitted when it is clear from the context. We begin with a common variational perspective in order to derive our higher order algorithms. In particular the Bregman Lagrangian (see \cite{Wibisono_2016}, Equation 1) is restated below as
\begin{equation*}
    \mathcal{L}(\theta,\dot{\theta},t)=\text{e}^{\bar{\alpha}_t+\bar{\gamma}_t}\left(D_h(\theta+\text{e}^{-\bar{\alpha}_t}\dot{\theta},\theta)-\text{e}^{\bar{\beta}_t}L(\theta)\right)
\end{equation*}
where $D_h$ is the Bregman divergence defined with a distance-generating function $h$ as: $D_h(y,x)=h(y)-h(x)-\left<\nabla h(x),y-x\right>$. This Lagrangian can be seen to weight the potential energy (loss) $L(\theta)$ versus kinetic energy $D_h(\theta+\text{e}^{-\bar{\alpha}_t}\dot{\theta},\theta)$, with a term $\text{e}^{\bar{\alpha}_t+\bar{\gamma}_t}$ which adjusts the damping. The user defined time-varying parameters ($\bar{\alpha}_t,\bar{\beta}_t,\bar{\gamma}_t$) will be defined in the following section.

\subsection{Time-varying regression}
\label{ss:Accelerated_Error1}
For ease of exposition, we will use the squared Euclidean norm $h(x)=\frac{1}{2}\lVert x\rVert^2$ in the Bregman divergence along with the squared loss $L=\frac{1}{2}e_y^2$ as was used in Section \ref{ss:TV_Reg}. The following are our choice of the time-varying scaling parameters: $\bar{\alpha}_t=\ln(\beta\N_t)$, $\bar{\beta}_t=\ln(\gamma/(\beta\N_t))$, and $\bar{\gamma}_t=\int_{t_0}^t\beta\N_{\nu}d\nu$, where $\gamma,\beta>0$ are scalar design parameters and
\begin{equation}
    \label{e:N_t}
    \N_t\triangleq(1+\mu\phi^T\phi)
\end{equation}
with scalar $\mu>0$ is a function of the time-varying feature, and is referred to as a \emph{normalizing signal}. It can be noticed that the second ``ideal scaling condition'' (Equation 2b, $\dot{\bar{\gamma}}_t=e^{\bar{\alpha}_t}$) of \cite{Wibisono_2016} holds but the first ``ideal scaling condition'' (Equation 2a, $\dot{\bar{\beta}}_t\leq e^{\bar{\alpha}_t}$) does not need to hold in general. In this sense, the results of this paper are applicable to a larger class of algorithms. With this choice of parameters, distance-generating function and loss function, the following non-autonomous Lagrangian results:
\begin{equation}
    \label{e:Lagrangian_Error1}
    \mathcal{L}(\theta,\dot{\theta},t)=\text{e}^{\int_{t_0}^t\beta\N_{\nu} d\nu}\frac{1}{\beta\N_t}\left(\frac{1}{2}\dot{\theta}^T\dot{\theta}-\gamma\beta\N_t\frac{1}{2}e_y^2\right).
\end{equation}
The Lagrangian in equation (\ref{e:Lagrangian_Error1}) is the central idea which will produce the first higher order algorithm in this paper. This Lagrangian is a function of not only the parameter and its time derivative, but is also a function of the time-varying feature $\phi$ directly through normalizing signal $\N_t$. Using the Lagrangian in (\ref{e:Lagrangian_Error1}), a functional may be defined as: $J(\theta)=\int_{\mathbb{T}}\mathcal{L}(\theta,\dot{\theta},t)dt$, where $\mathbb{T}$ is a time interval. To minimize this functional, a necessary condition from the calculus of variations is that the Lagrangian solves the Euler-Lagrange equation \cite{Luenberger_1969}: $\frac{d}{dt}\left(\frac{\partial\mathcal{L}}{\partial\dot{\theta}}(\theta,\dot{\theta},t)\right)=\frac{\partial\mathcal{L}}{\partial\theta}(\theta,\dot{\theta},t)$. The second order differential equation resulting from the application of the Euler-Lagrange equation is:
\begin{equation}
    \label{e:Accelerated_GF_2nd}
    \ddot{\theta}+\left[\beta\N_t-\frac{\dot{\N}_t}{\N_t}\right]\dot{\theta}=-\gamma\beta\N_t\phi e_y.
\end{equation}
Here $\beta$ can be seen to adjust ``friction''. Taking $\beta\rightarrow\infty$ (strong friction limit) results in the standard first order algorithm (\ref{e:update_GF}).\footnote{\label{fn:Beta} This notion will be more rigorously shown in Section \ref{s:Stability}.} The second order differential equation in (\ref{e:Accelerated_GF_2nd}) may be implemented using two first-order differential equations, similar to \cite{Evesque_2003}:
\begin{equation}
    \label{e:Accelerated_GF}
    \dot{\vartheta}=-\gamma\nabla_{\theta}L=-\gamma\phi e_y,\qquad \dot{\theta}=-\beta(\theta-\vartheta)\N_t.
\end{equation}
Equation (\ref{e:Accelerated_GF}) is the higher order algorithm that represents the first main contribution of this paper. It can be seen that the first of the two equations in (\ref{e:Accelerated_GF}) is identical to the first-order update (\ref{e:update_GF}); the second equation can be viewed as a filter, normalized by the feature-dependent $\N_t$. Alternately, the first equation can be viewed as a gradient step and the second as a mixing step. Similar in form to batch normalization \cite{Ioffe_2015} and the update in ADAM \cite{Kingma_2017}, the normalization present in this algorithm is different in that it normalizes by the time-varying feature itself as opposed to estimated moments. In equation (\ref{e:Accelerated_GF}) it can be seen that $\beta\rightarrow\infty$ decreases the nominal time constant (for a given $\phi$) of the time-varying filter, and thus in the limit $\beta\rightarrow\infty$ the first order algorithm (\ref{e:update_GF}) is recovered.\footnoteref{fn:Beta}

\subsection{Model reference adaptive control and identification}
\label{ss:Accelerated_Error2}

In a similar manner to (\ref{e:Lagrangian_Error1}) the following non-autonomous Lagrangian is defined:
\begin{equation}
    \label{e:Lagrangian_Error2}
    \mathcal{L}(\theta,\dot{\theta},t)=\text{e}^{\int_{t_0}^t\beta\N_{\nu} d\nu}\frac{1}{\beta\N_t}\left(\frac{1}{2}\dot{\theta}^T\dot{\theta}-\gamma\beta\N_t\left[\frac{d}{dt}\left\{\frac{e^TPe}{2}\right\}+\frac{e^TQe}{2}\right]\right).
\end{equation}
Comparing the Lagrangian in (\ref{e:Lagrangian_Error2}) to that in (\ref{e:Lagrangian_Error1}), it can be seen that they only differ by the term in the square brackets, representing the loss function considered. The extra terms in the square brackets account for energy storage in the error model dynamics in equation (\ref{e:error_model_2}). This may be seen as: $\left[\frac{d}{dt}\left\{\frac{e^TPe}{2}\right\}+\frac{e^TQe}{2}\right]=e^TP(\dot{e}-Ae)=e^TPb\tilde{\theta}^T\phi$, where the loss is only zero for this dynamical error model when both $e$ and $\dot{e}$ are zero. Using this Lagrangian, a functional may be defined as $J(\theta)=\int_{\mathbb{T}}\mathcal{L}(\theta,\dot{\theta},t)dt$. The minimization of this functional with the Euler-Lagrange equation and error dynamics (\ref{e:error_model_2}) results in the following second order differential equation:
\begin{equation}
    \label{e:Accelerated_MRAC_2nd}
    \ddot{\theta}+\left[\beta\N_t-\frac{\dot{\N}_t}{\N_t}\right]\dot{\theta}=-\gamma\beta\N_t\phi e^TPb.
\end{equation}
Again, $\beta$ can be seen to represent ``friction'', with $\beta\rightarrow\infty$ resulting in the first order algorithm (\ref{e:MRAC}).\footnoteref{fn:Beta} The second order differential equation in (\ref{e:Accelerated_MRAC_2nd}) may be implemented using two first-order differential equations, similar to \cite{Evesque_2003}:
\begin{equation}
    \label{e:Accelerated_MRAC}
    \dot{\vartheta}=-\gamma\phi e^TPb,\qquad \dot{\theta}=-\beta(\theta-\vartheta)\N_t.
\end{equation}
Equation (\ref{e:Accelerated_MRAC}) is the higher order algorithm that represents the second main contribution of this paper. Similar to (\ref{e:Accelerated_GF}), the first equation may be viewed as the stability based update (\ref{e:MRAC}); the second equation may be viewed as a filter, normalized by the feature-dependent $\N_t$.

\section{Stability analysis and regret bounds}
\label{s:Stability}

In this section we state the main stability and convergence results, as well as regret bounds for the higher order algorithms derived in this paper. The class $\mathcal{L}_p$ is described in Definition \ref{d:L_p} of Appendix \ref{s:Definitions}. Unless otherwise specified, $\lVert\cdot\rVert$ represents the 2-norm. Stability analysis using Lyapunov functions have been of increased use in recent years in state of the art machine learning approaches \cite{Wilson_2016,Wilson_2018}. A brief overview is given in Appendix \ref{ss:Stability_Lyapunov}.

\begin{theorem}[Time-varying regression]\label{th:Stability_Error1}
For the higher order algorithm in (\ref{e:Accelerated_GF}) applied to the time-varying regression problem in (\ref{e:error1}) the following
\begin{equation}
    V=\frac{1}{\gamma}\lVert\vartheta-\theta^*\rVert^2+\frac{1}{\gamma}\lVert\theta-\vartheta\rVert^2
\end{equation}
is a Lyapunov function with time derivative $\dot{V}\leq-\frac{2\beta}{\gamma}\lVert\theta-\vartheta\rVert^2-\lVert e_y\rVert^2-\left[\lVert e_y\rVert-2\lVert\theta-\vartheta\rVert\lVert\phi\rVert\right]^2\leq0$ and therefore $(\vartheta-\theta^*)\in\mathcal{L}_{\infty}$ and $(\theta-\vartheta)\in\mathcal{L}_{\infty}$. If in addition it assumed that $\phi,\dot{\phi} \in \mathcal{L}_{\infty}$ then  $\lim_{t\rightarrow\infty}e_y(t)=0$, $\lim_{t\rightarrow\infty}(\theta(t)-\vartheta(t))=0$, $\lim_{t\rightarrow\infty}\dot{\vartheta}(t)=0$, and $\lim_{t\rightarrow\infty}\dot{\tilde{\theta}}(t)=0$.
\end{theorem}
\begin{corollary}\label{c:Stability_Error1}
The higher order algorithm in \eqref{e:Accelerated_GF} applied to the time-varying regression problem in \eqref{e:error1} has uniformly bounded regret: $\text{Regret}_\mathrm{continuous}:= \int_0^T \lVert e_y(\tau)\rVert^2d\tau = \mathcal{O}(1)$.
\end{corollary}
\begin{theorem}[Model reference adaptive control]\label{th:Stability_Error2}
For the higher order algorithm in (\ref{e:Accelerated_MRAC}) applied to the model reference adaptive control problem in (\ref{e:error_model_2}) the following
\begin{equation}
    V=\frac{1}{\gamma}\lVert\vartheta-\theta^*\rVert^2+\frac{1}{\gamma}\lVert\theta-\vartheta\rVert^2+e^TPe
\end{equation}
is a Lyapunov function with time derivative $\dot{V}\leq-\frac{2\beta}{\gamma}\lVert\theta-\vartheta\rVert^2-\lVert e\rVert^2-\left[\lVert e\rVert-2\lVert Pb\rVert\lVert\theta-\vartheta\rVert\lVert\phi\rVert\right]^2\leq0$ and therefore $e\in\mathcal{L}_{\infty}$, $(\vartheta-\theta^*)\in\mathcal{L}_{\infty}$, and $(\theta-\vartheta)\in\mathcal{L}_{\infty}$. If in addition it assumed that $\phi\in \mathcal{L}_{\infty}$ then  $\lim_{t\rightarrow\infty}e(t)=0$. Also if $\dot{\phi}\in\mathcal{L}_{\infty}$, then $\lim_{t\rightarrow\infty}(\theta(t)-\vartheta(t))=0$, $\lim_{t\rightarrow\infty}\dot{\vartheta}(t)=0$, and $\lim_{t\rightarrow\infty}\dot{\tilde{\theta}}(t)=0$.
\end{theorem}
\begin{corollary}\label{c:Stability_Error2}
The higher order algorithm in (\ref{e:Accelerated_MRAC}) applied to the adaptive control problem in (\ref{e:error_model_2}) has uniformly bounded regret: $\text{Regret}_\mathrm{continuous}:= \int_0^T \lVert e(\tau)\rVert^2d\tau = \mathcal{O}(1)$.
\end{corollary}
For a proofs of Theorems \ref{th:Stability_Error1} and \ref{th:Stability_Error2} see Appendices \ref{ss:Stability_Error1} and \ref{ss:Stability_Error2} respectively. Corollaries \ref{c:Stability_Error1} and \ref{c:Stability_Error2} follow from $\dot{V}(t)\leq-\lVert e_y(t) \rVert^2$, $\dot{V}(t)\leq-\lVert e(t) \rVert^2$ and $V$ is bounded, as shown in Appendix \ref{ss:Regret}.

\section{Comparison of approaches}
\label{s:Approach_Comparison}

Table \ref{t:Comparison_Alg} shows a comparison of the Lagrangian functional and the resulting second order ODE from a given parameterization of the algorithm proposed by \cite{Wibisono_2016} to the results provided in this paper for regression. This parameterization was chosen to coincide with the notions used in this paper (squared loss and the squared Euclidean norm for the error in (\ref{e:error1})) along with parameters chosen as in Equation 12 of \cite{Wibisono_2016}. It can be seen that both Lagrangians have an increasing function multiplying the kinetic and potential energies with an additional time-varying term weighting the potential energy. Our approach however is a function of the feature $\phi$ as compared to an explicit function of time. This more natural parameterization results in an algorithm shown for comparison purposes in Table \ref{t:Comparison_Alg} that does not have a damping term that decays to zero with time. Therefore our algorithm does not change from an overdamped to underdamped system as time progresses as is commonly seen in higher order accelerated methods \cite{Su_2016}. Thus our approach is capable of running continuously as features are processed. No restart is required as is often used in accelerated algorithms in machine learning \cite{O_Donoghue_2013}. The more natural damping term shown in our higher order ODE is an explicit function of both the feature and time derivative of the feature vector. It can be noted once more that the time derivative of the feature does not need to be known as this ODE may be implemented as the higher order algorithm in (\ref{e:Accelerated_GF}), which allows for online processing of the features, without a priori knowledge of its future variation. Therefore the higher order algorithms derived in this paper can be used in real-time sequential decision making systems, where features and output errors are processed online.
\begin{table}[t]
\caption{Comparison of approaches for regression}
\label{t:Comparison_Alg}
\centering
\begin{tabular}{cc} 
\toprule
Parameterization from \cite{Wibisono_2016} & Our Approach \\ 
\midrule
$\mathcal{L}(\theta,\dot{\theta},t)=\frac{t^{p+1}}{p}\left(\frac{1}{2}\dot{\theta}^T\dot{\theta}-Cp^2t^{p-2}\frac{1}{2}e_y^2\right)$ & $\mathcal{L}(\theta,\dot{\theta},t)=\text{e}^{\int_{t_0}^t\beta\N_{\nu} d\nu}\frac{1}{\beta\N_t}\left(\frac{1}{2}\dot{\theta}^T\dot{\theta}-\gamma\beta\N_t\frac{1}{2}e_y^2\right)$\\
$\ddot{\theta}+\frac{p+1}{t}\dot{\theta}=-Cp^2t^{p-2}\phi e_y$ & $\ddot{\theta}+\left[\beta\N_t-\frac{\dot{\N}_t}{\N_t}\right]\dot{\theta}=-\gamma\beta\N_t\phi e_y$\\
\bottomrule
\end{tabular}
\end{table}
\begin{table}[t]
\caption{Comparison of candidate Lyapunov functions for the higher order regression algorithm}
\label{t:Comparison_Lyap}
\centering
\begin{tabular}{cc} 
\toprule
Lyapunov Function in \cite{Wibisono_2016} & Our Approach \\ 
\midrule
$V=\frac{1}{2}\lVert\tilde{\theta}+\frac{1}{\beta\N_t}\dot{\theta}\rVert^2+\frac{\gamma}{\beta\N_t}\frac{1}{2}e_y^2$ & $V=\frac{1}{\gamma}\lVert\vartheta-\theta^*\rVert^2+\frac{1}{\gamma}\lVert\theta-\vartheta\rVert^2$ \\ 
$\dot{V}=-\gamma e_y^2\left(1+\frac{\mu\phi^T{\color{red}\dot{\phi}}}{\beta\N_t^2}\right)+\frac{\gamma}{\beta\N_t}e_y\tilde{\theta}^T{\color{red}\dot{\phi}}$ & \hspace{-0.068cm}$\dot{V}\leq-\frac{2\beta}{\gamma}\lVert\theta-\vartheta\rVert^2-\lVert e_y\rVert^2-\left[\lVert e_y\rVert-2\lVert\theta-\vartheta\rVert\lVert\phi\rVert\right]^2$ \\
\bottomrule
\end{tabular}
\end{table}

Normalization by the magnitude of the time-varying feature (\ref{e:N_t}) can be seen to be explicitly included in our algorithm. This normalization is in fact necessary in order to provide a proof of stability as was found by \cite{Evesque_2003}, due to the required time-varying feature dependent scaling. Table \ref{t:Comparison_Lyap} shows the candidate Lyapunov function proposed by \cite{Wibisono_2016} applied to our algorithm as derived in Appendix \ref{ss:Stability_Wibisono}, and the Lyapunov function considered in this paper. It can be seen that the candidate Lyapunov function proposed by \cite{Wibisono_2016} represents a scaled kinetic plus potential energy, and results in a time derivative that cannot be guaranteed to be non-increasing for arbitrary initial conditions and time variations of the feature. Our Lyapunov function is fundamentally different in its construction and is indeed able to verify stability. It should be noted that the class of algorithms in \cite{Wibisono_2016} was not designed for time-varying features and that the comparisons are due to its general form in continuous time, representing a large class of higher order learning algorithms commonly used in machine learning, including Nesterov acceleration \cite{Nesterov_1983}. It can also be noted that the higher order algorithms proposed in this paper are \emph{proven stable regardless of the initial condition} of the system (see Section \ref{s:Stability}). That is to say that an optimization problem-specific schedule on the parameters of the problem is not required to set in order to cope with the initial conditions of the algorithm, as is usually required for momentum methods commonly used in machine learning \cite{Sutskever_2013}. Our regret bounds do not increase as a function of time as is common in online machine learning approaches \cite{Zinkevich_2003,Hazan_2007,Hazan_2008,Shalev_Shwartz_2011,Hazan_2016}. Thus $\mathcal{O}(1)$, constant regret attained by our algorithms is the best achievable regret, up to constants which do not vary with time. The provably correct algorithms proposed in this paper are proven to be stable, with $\mathcal{O}(1)$ regret bounds, and provide for a unified framework using a variational perspective for convergence in output (\ref{e:Accelerated_GF}) (respectively model tracking (\ref{e:Accelerated_MRAC})) error for \emph{time-varying features with arbitrary initial conditions} where the relation between feature and error may be algebraic (\ref{e:error1}) or dynamical (\ref{e:error_model_2}).

\section{Numerical experiments}
\label{s:Experiments}

We conducted numerical experiments for the time-varying regression and state feedback adaptive control problems. The implementation was carried out in Matlab and Simulink, in order to efficiently simulate continuous dynamical systems (code provided). The hyperparameters for the higher order algorithms were chosen using nominal values (e.g. $\gamma=0.1$, $\beta=1$, and $\mu$ selected as in the proof of stability in Appendices \ref{ss:Stability_Error1} and \ref{ss:Stability_Error2}) as opposed to optimized for performance, to demonstrate the efficacy of the algorithms without the need for significant hyperparameter tuning.

\subsection{Time-varying regression}
\label{ss:TV_Regression_Experiments}

A time-varying regression system was simulated with the error model as in (\ref{e:error1}). The standard gradient flow algorithm (\ref{e:update_GF}) and higher order algorithm (\ref{e:Accelerated_GF}) are compared in each simulation alongside a continuous parameterization of Nesterov's accelerated method as shown in Table \ref{t:Comparison_Alg}. A three dimensional problem was considered for the sake of clarity of presentation. For the accelerated method in \cite{Wibisono_2016}, the hyperparameters were set to correspond to Nesterov acceleration and to have the same constant multiplying the gradient term, $\phi e_y$, as: $p=2$, $C=\gamma\beta/p^2$.

\begin{figure}[t!]
    \centering
    \begin{subfigure}[b]{\textwidth}
        \centerline{
	    \includegraphics[trim={0.25cm 3.5cm 1.5cm 4.4cm},clip,width=0.25\textwidth]{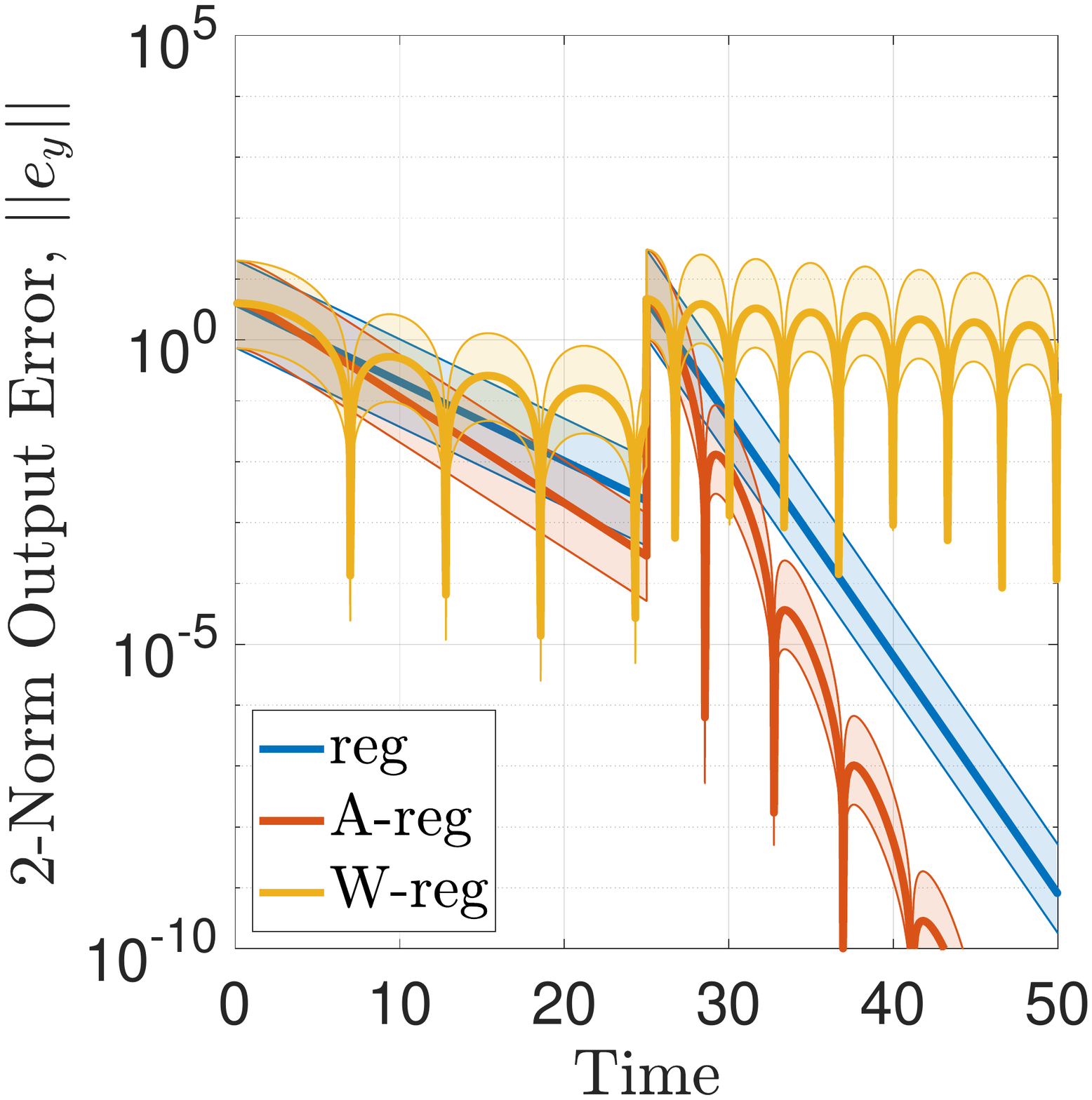}
	    \includegraphics[trim={0.25cm 3.5cm 1.5cm 4.4cm},clip,width=0.25\textwidth]{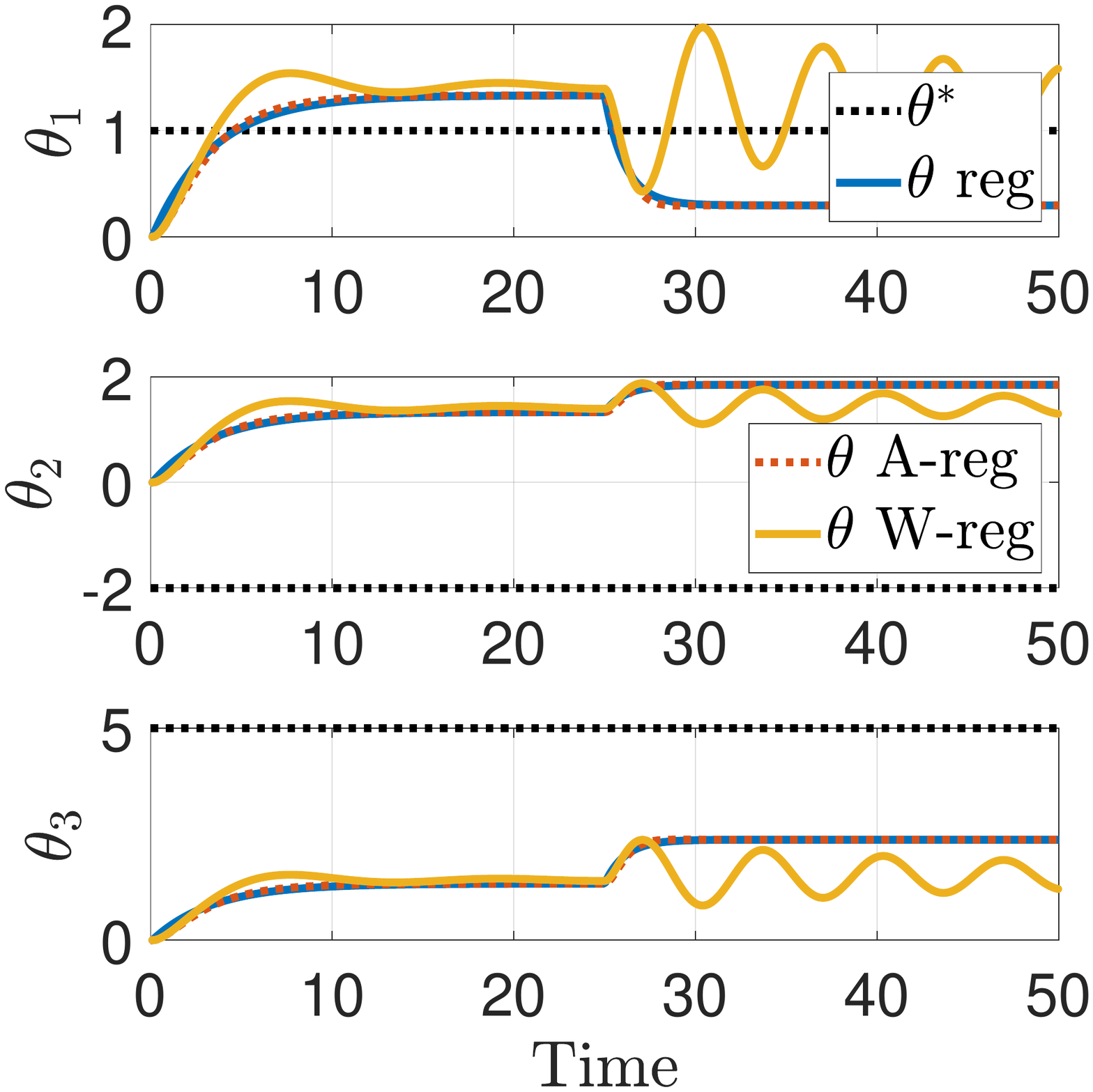}
	    \includegraphics[trim={0.25cm 3.5cm 1.5cm 4.4cm},clip,width=0.25\textwidth]{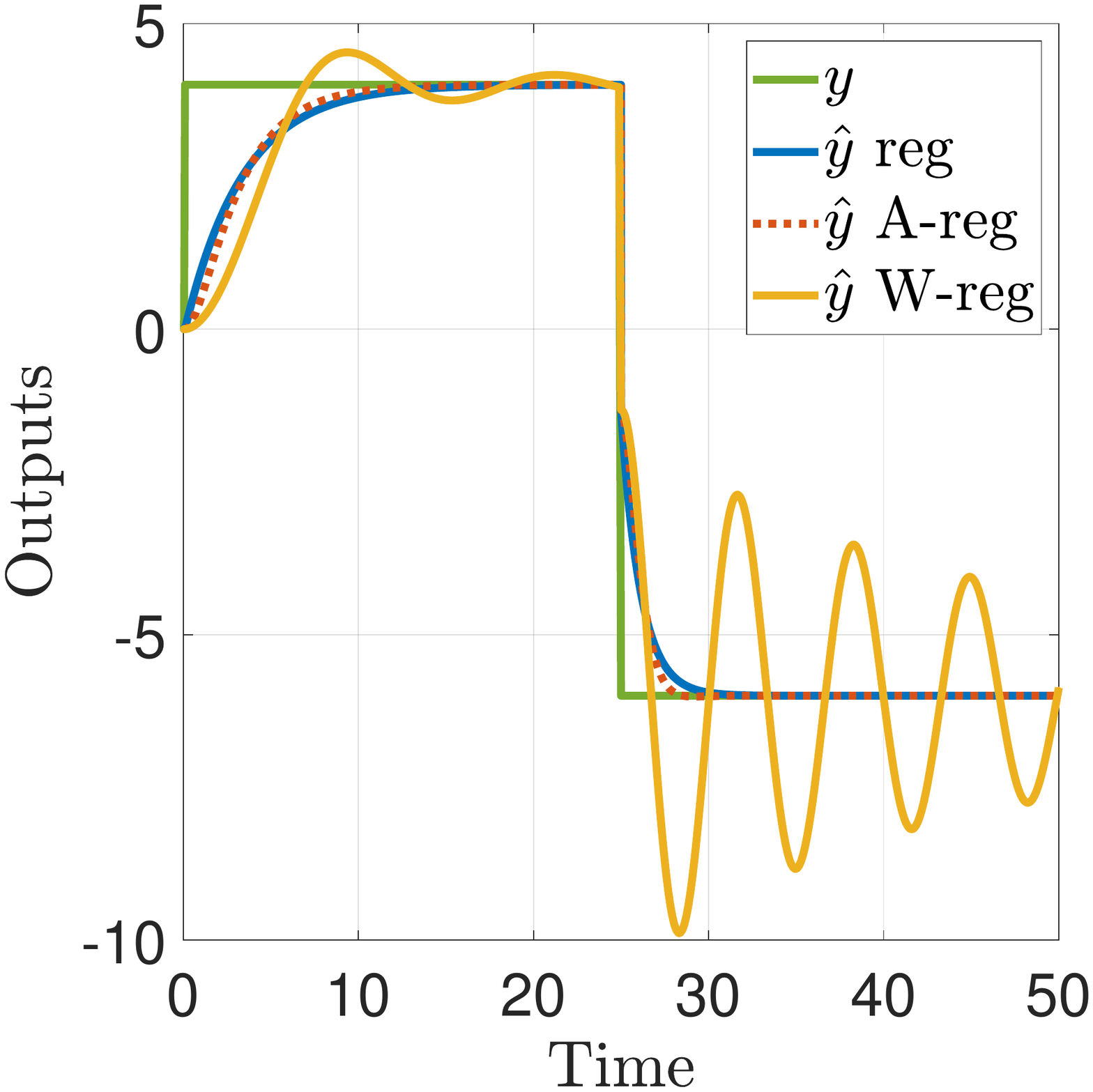}
	    \includegraphics[trim={0.25cm 3.5cm 1.5cm 4.4cm},clip,width=0.25\textwidth]{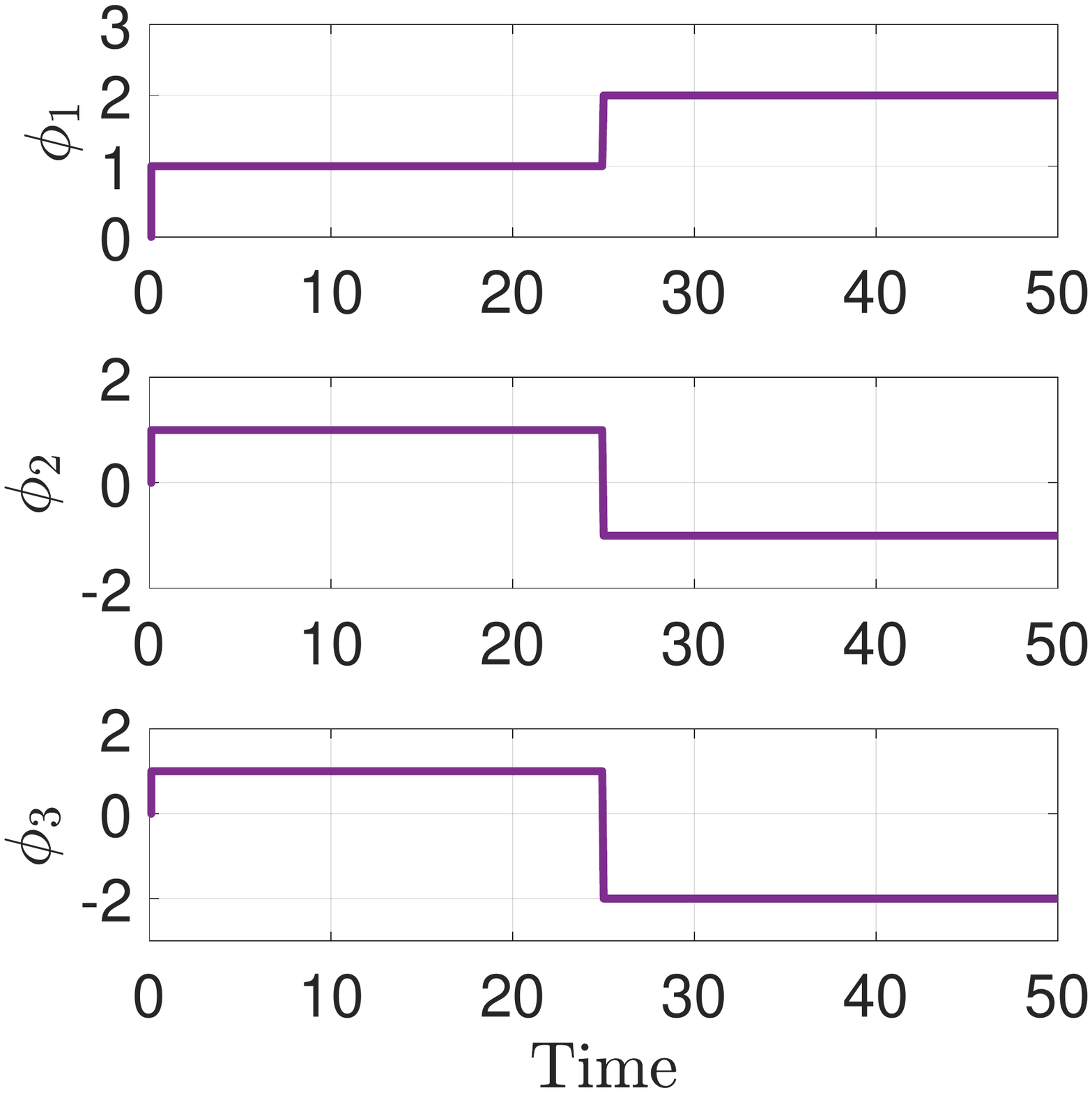}
	    }
        \caption{Time-varying regression: $\theta^*=[1,-2,~5]^T+Z$, where $Z\sim\text{Unif}\left([-10,10]^3\right)$. At time $t=0.1$, the feature vector steps to a constant value of $\phi=[1,~1,~1]^T$, at time $t=25$, the feature vector steps to $\phi=[2,-1,-2]^T$.}
        \label{f:Error1_2step_Response}
    \end{subfigure}
    
    \begin{subfigure}[b]{\textwidth}
        \centerline{
	    \includegraphics[trim={0.25cm 3.5cm 1.5cm 4.4cm},clip,width=0.25\textwidth]{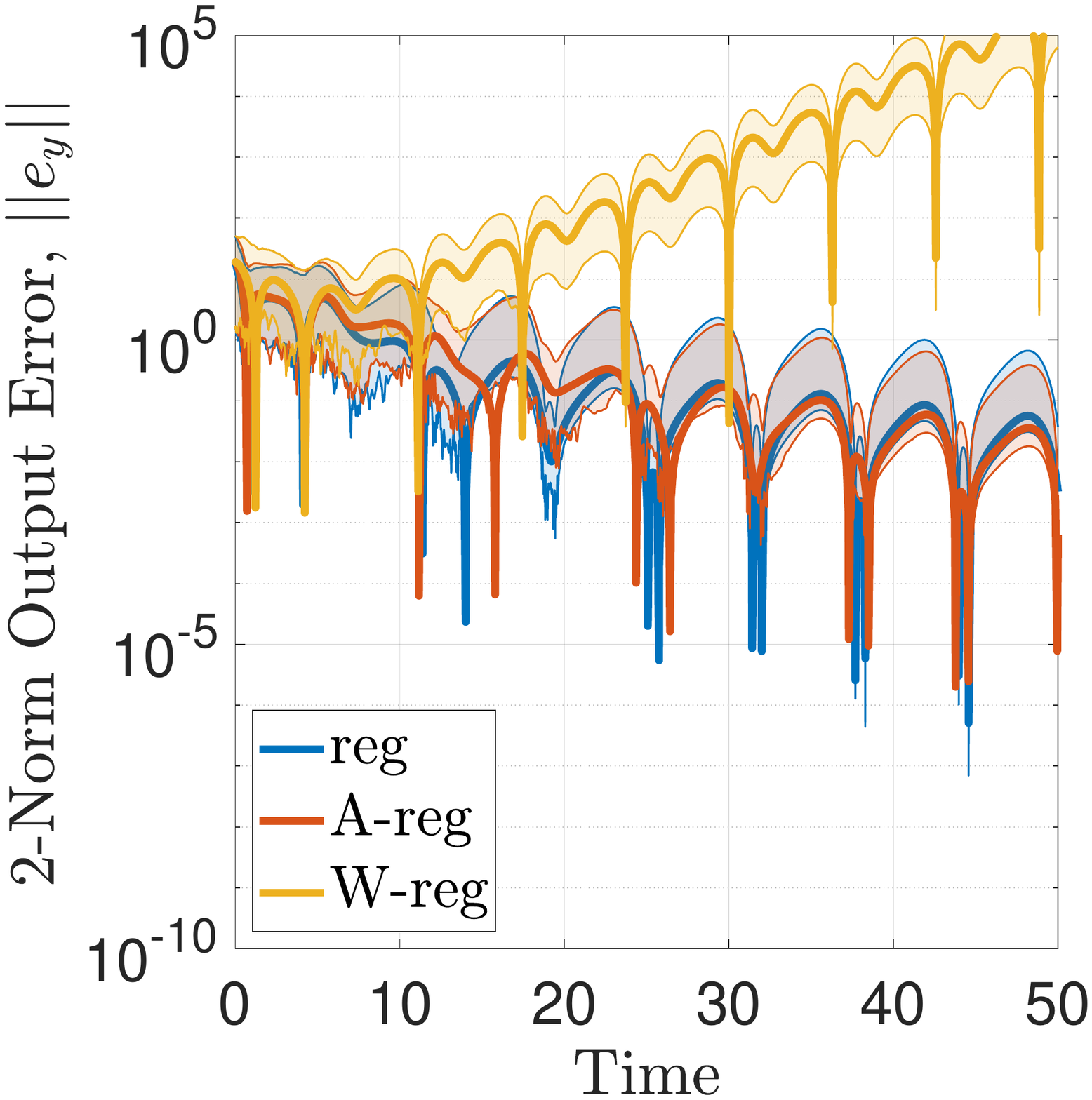}
	    \includegraphics[trim={0.25cm 3.5cm 1.5cm 4.4cm},clip,width=0.25\textwidth]{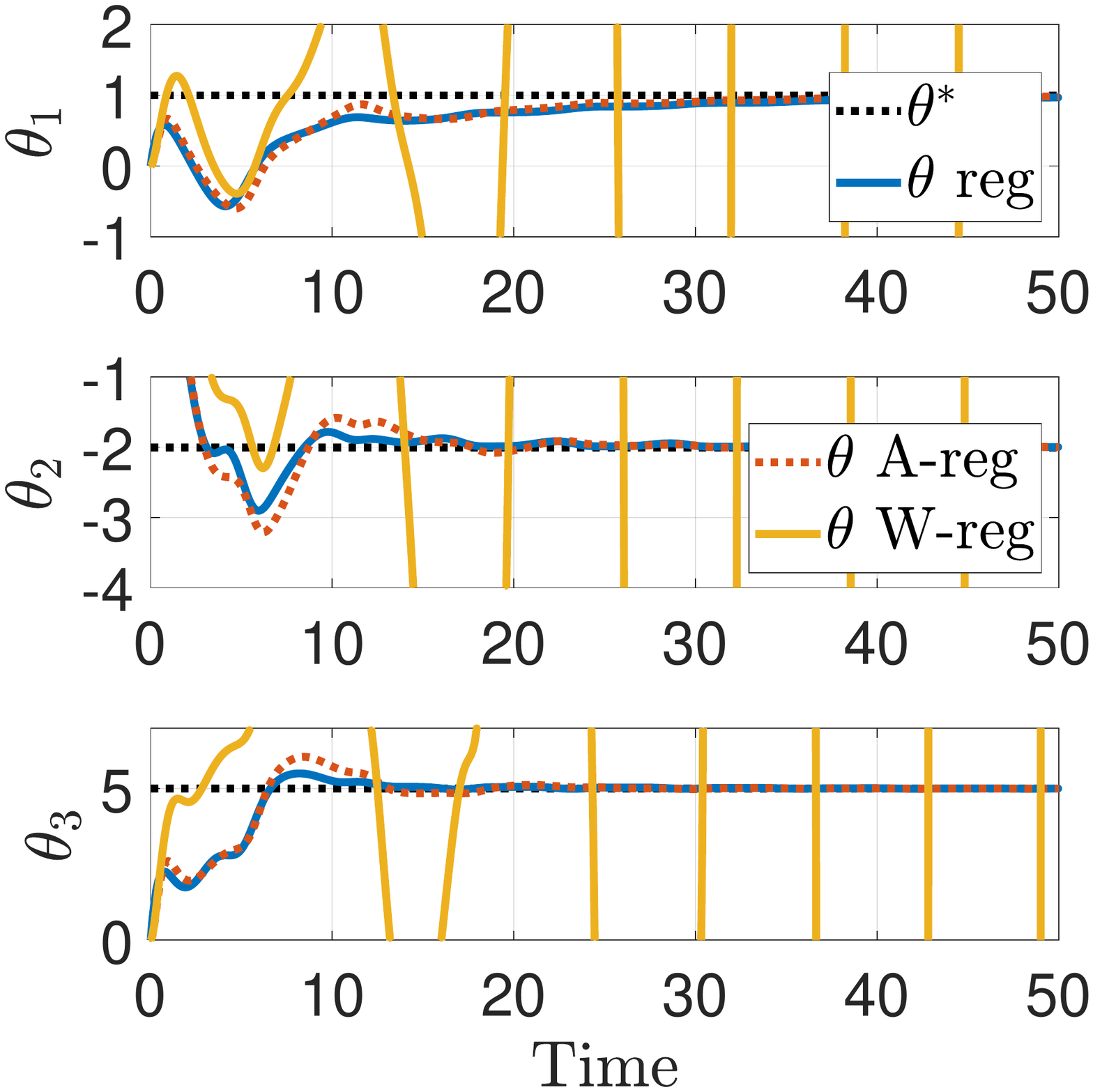}
	    \includegraphics[trim={0.25cm 3.5cm 1.5cm 4.4cm},clip,width=0.25\textwidth]{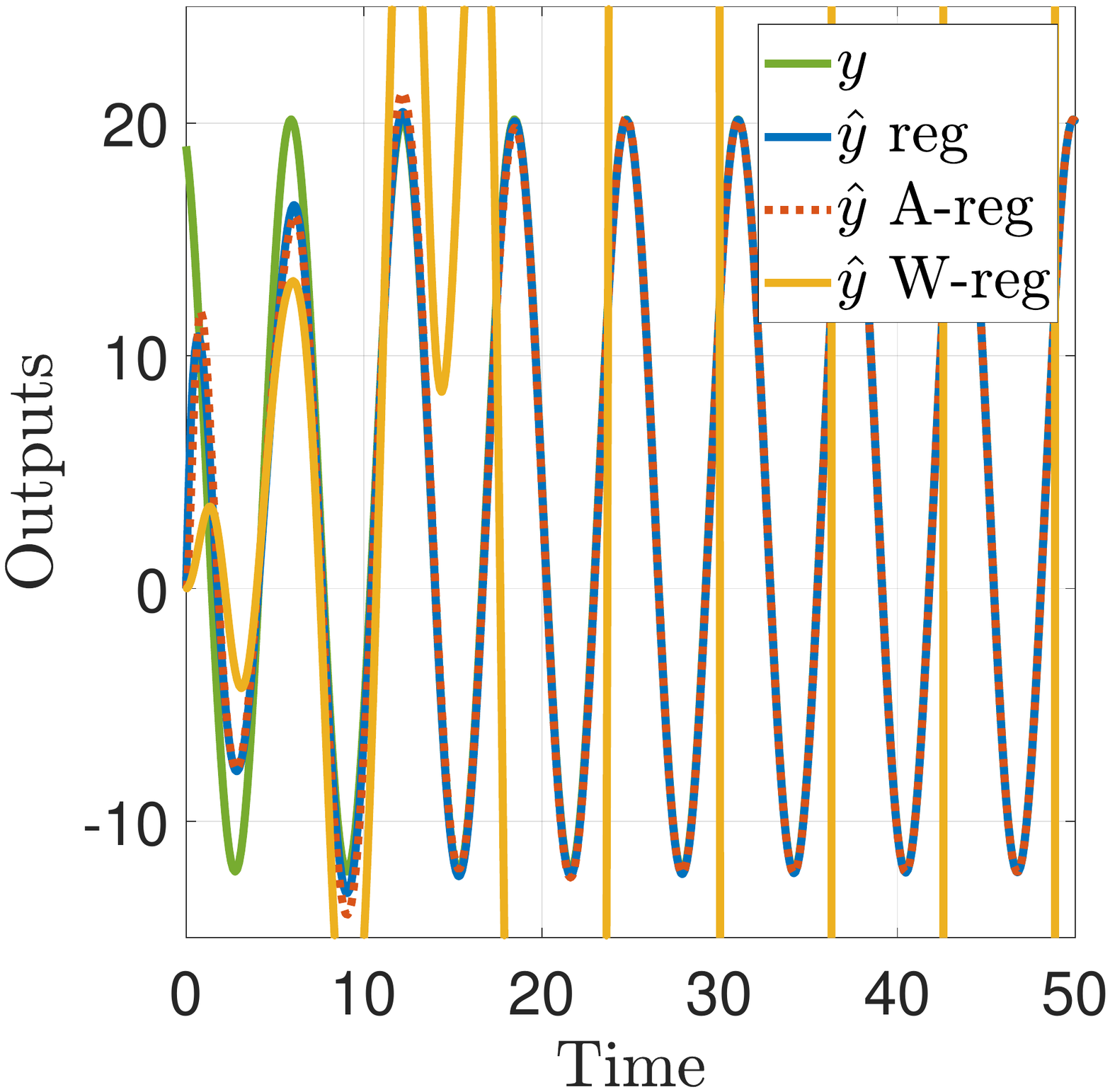}
	    \includegraphics[trim={0.25cm 3.5cm 1.5cm 4.4cm},clip,width=0.25\textwidth]{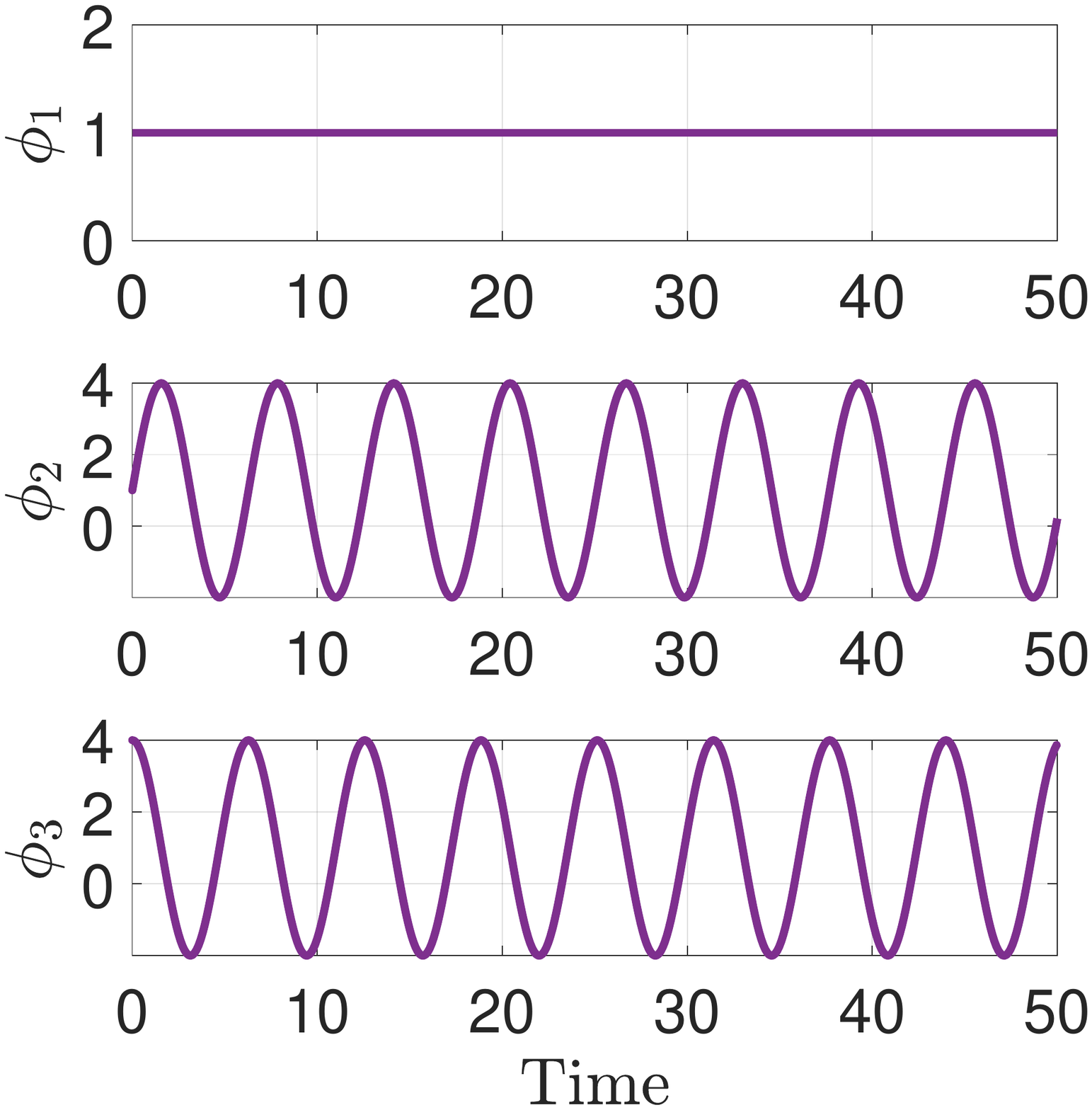}
	    }
        \caption{Time-varying regression: $\theta^*=[1,-2,~5]^T+Z$, where $Z\sim\text{Unif}\left([-10,10]^3\right)$. The feature response is persistently exciting (PE) with $\phi=[1,~1+3\sin(t),~1+3\cos(t)]^T$.}
        \label{f:Error1_PE_Response}
    \end{subfigure}
    
    \begin{subfigure}[b]{\textwidth}
        \centerline{
	    \includegraphics[trim={0.25cm 3.5cm 1.5cm 4.4cm},clip,width=0.25\textwidth]{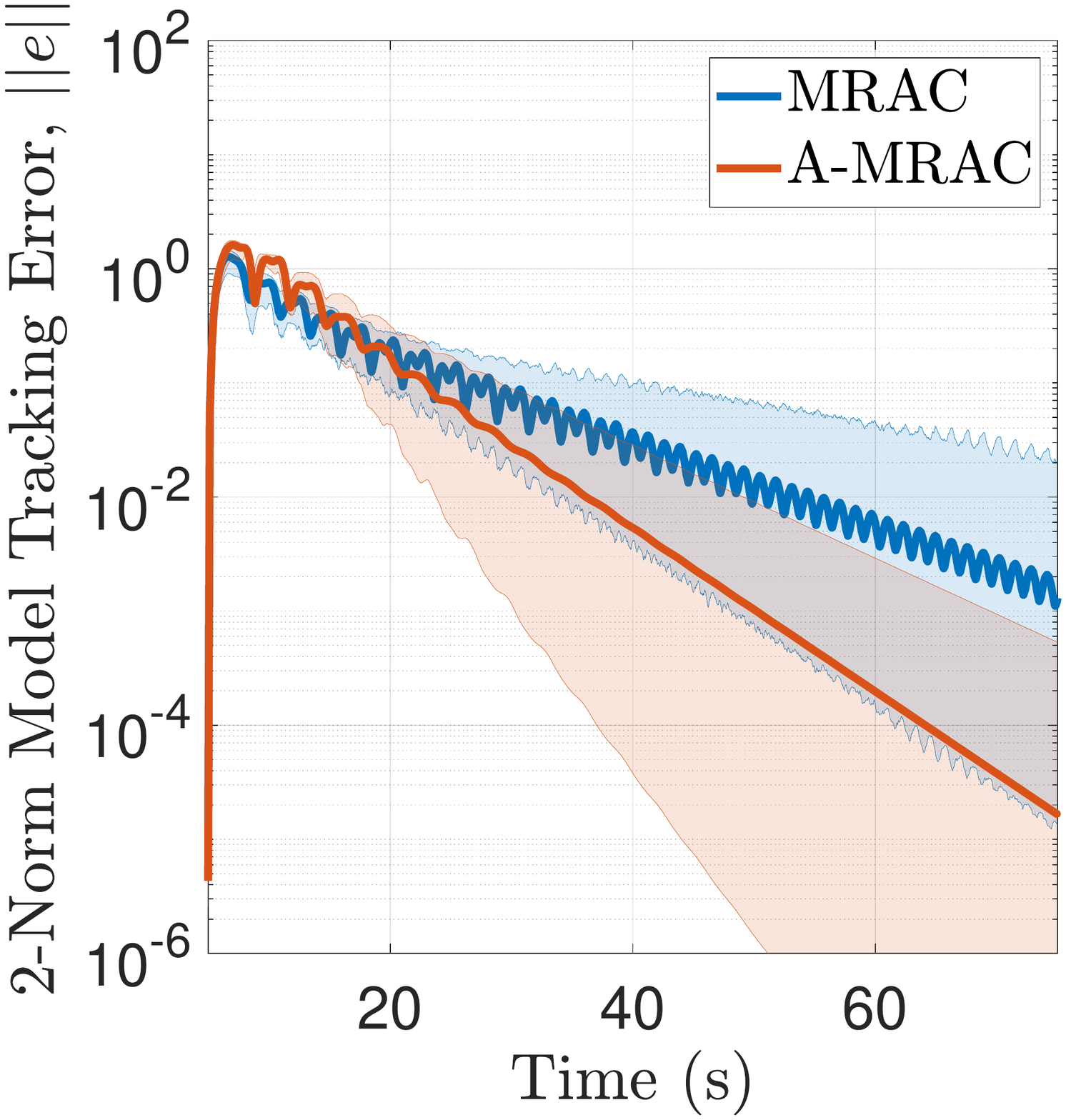}
	    \includegraphics[trim={0.25cm 3.5cm 1.5cm 4.4cm},clip,width=0.25\textwidth]{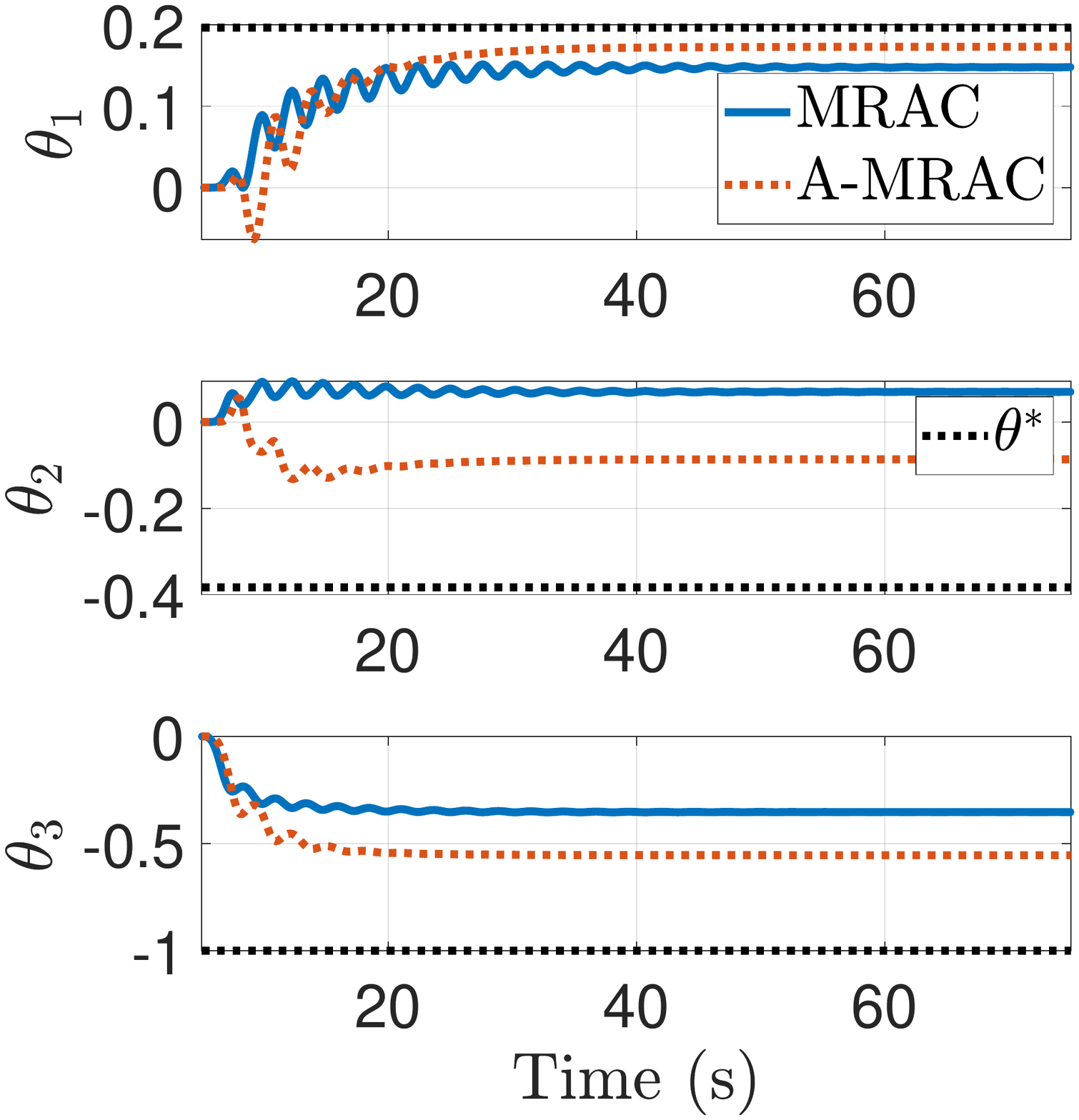}
	    \includegraphics[trim={0.25cm 3.5cm 1.5cm 4.4cm},clip,width=0.25\textwidth]{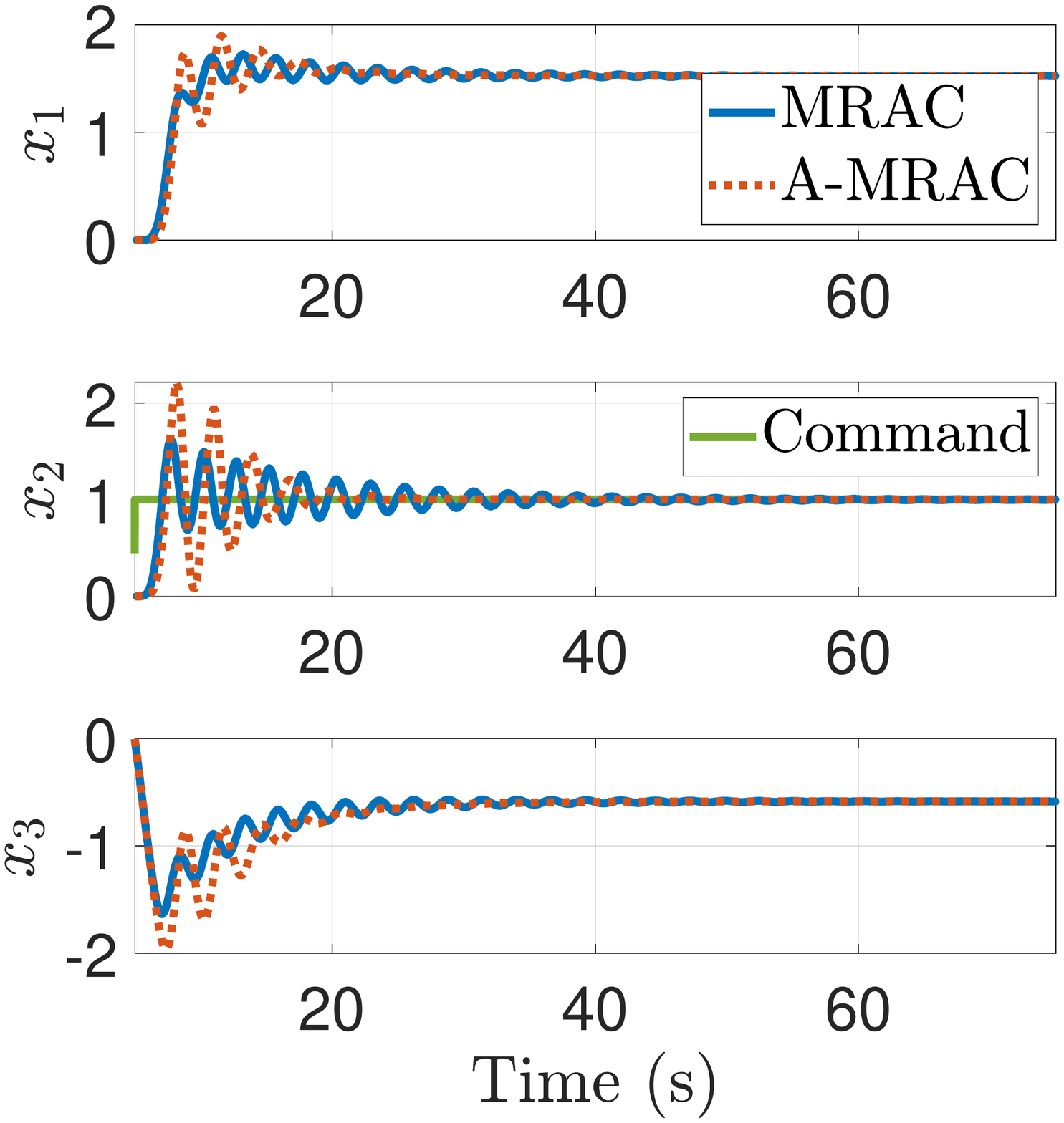}
	    \includegraphics[trim={0.25cm 3.5cm 1.5cm 4.4cm},clip,width=0.25\textwidth]{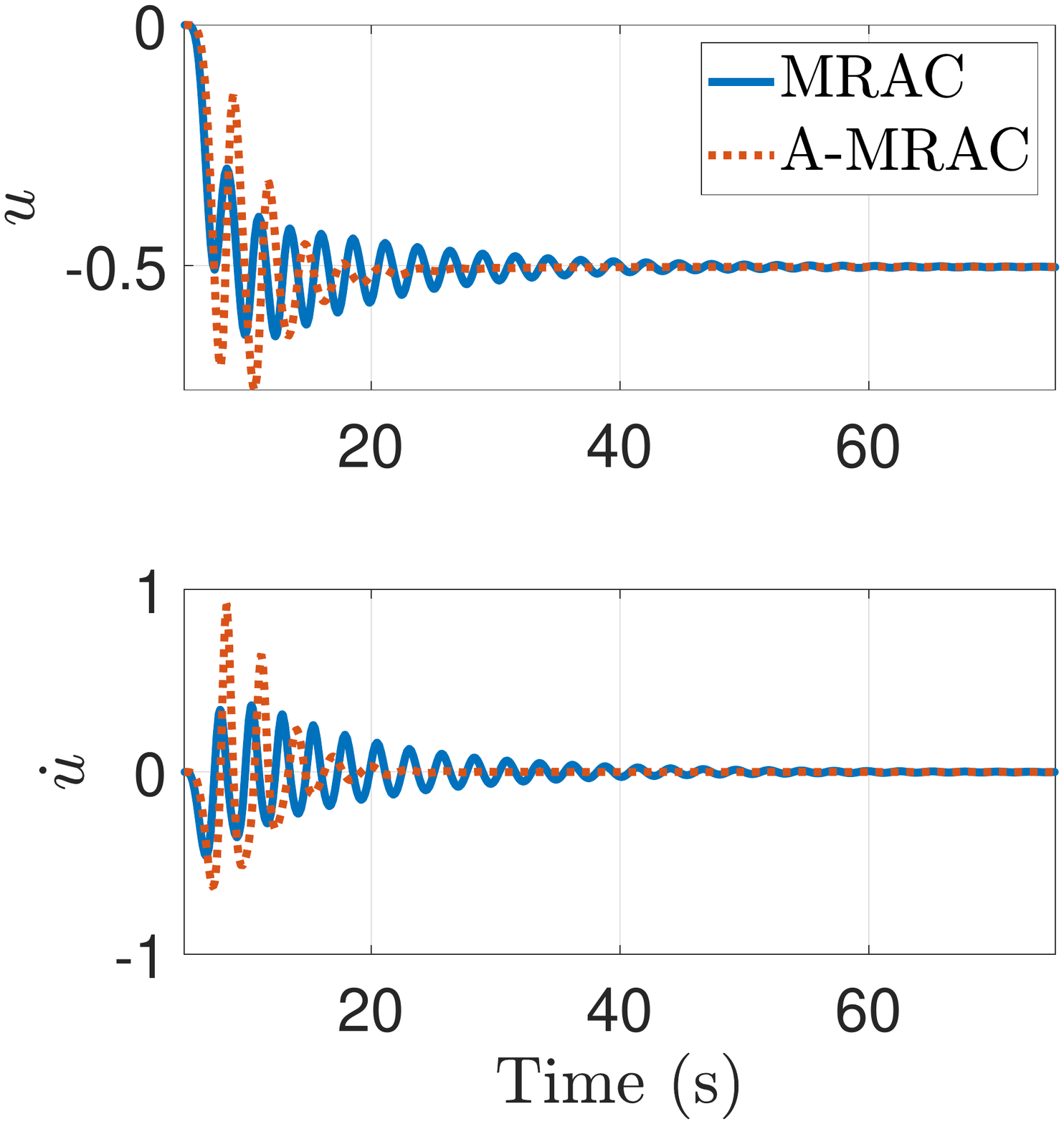}
	    }
        \caption{Adaptive control: $\theta^*=[0.1965,-0.3835,-1]^T\cdot W$, where $W\sim\text{Unif}\left([-1/2,2]\right)$. At time $t=5$, the command for the state $x_2$ to track changes to a value of $1$. Consequently, the states change to track the command.}
        \label{f:Error_2Step_Response}
    \end{subfigure}
    \caption{(to be viewed in color) Left plot: Output and model tracking error trajectories. Left-middle: Parameter trajectories. Right-middle: Output and state trajectories. Right: Time-varying features and input trajectories. $95\%$ intervals for error plots shown as shaded regions. Example trajectories shown as solid and dashed lines.}
    \label{f:figure_collection}
\end{figure}

For the first simulation shown in Figure \ref{f:Error1_2step_Response}, the feature vector was initially set equal to the zero with initial conditions of all algorithms initialized at zero (consistent with not knowing the feature variation and unknown parameter ahead of time). The gradient flow algorithm for regression (\ref{e:update_GF}) (denoted ``reg'') as well as higher order algorithms (denoted ``A-reg'' for the algorithm presented in this paper (\ref{e:Accelerated_GF}) and ``W-reg'' for the algorithm by \cite{Wibisono_2016}, parameterized in Table \ref{t:Comparison_Alg}) are seen to converge in output. The higher order algorithm presented in this paper (\ref{e:Accelerated_GF}) can be seen to converge at a faster rate however. As a separate note, given that the system does not have a persistently exciting feature $\phi$ (Appendix \ref{s:Definitions}, Definition \ref{d:PE}); the parameter $\theta$ does not converge to the true value.

Figure \ref{f:Error1_PE_Response} shows the response for persistently exciting features, which can be seen to consist of time-varying functions as compared to constant feature steps of Figure \ref{f:Error1_2step_Response}. Additional plots in Figure \ref{f:figure_collection2} of Appendix \ref{s:Experiment_MRAC_Implementation} show a progression in the increase of the time variation of the feature vector. It can be seen that as the time variation of the feature increases, the ``W-reg'' algorithm modeling Nesterov acceleration in continuous time becomes \emph{unstable}, as shown in the left plot of Figure \ref{f:Error1_PE_Response}. A similar destabilizing effect can occur for many higher order accelerated algorithms commonly employed in the machine learning community when features are time-varying. Our provably correct higher order learning algorithm in (\ref{e:Accelerated_GF}) maintains stability and convergence despite the feature time-variation. It should also be noted that the feature profile variation considered here is persistently exciting. Thus in addition to output error tending towards zero, as proved in Section \ref{ss:Stability_Error1}, the error in the parameter space can additionally be seen to tend towards zero (i.e. $\theta\rightarrow\theta^*$). The step changes and sinusoidal time-varying feature profiles considered in this section were chosen to be representative of model/concept shift in machine learning problems \cite{Gama_2014}, as well as system identification problems \cite{Ljung_1987}.

\subsection{State feedback adaptive control}
\label{ss:Error_2_Experiments}

A state feedback model reference adaptive control (MRAC) problem was simulated with the error model in (\ref{e:error_model_2}). The standard MRAC algorithm (\ref{e:MRAC}) and higher order algorithm (\ref{e:Accelerated_MRAC}) were compared in a simulation of linearized longitudinal dynamics of an F-16 aircraft with integral command tracking. More details regarding the simulation implementation can be found in Appendix \ref{s:Experiment_MRAC_Implementation}, including an explanation for the choice of the unknown parameter and definitions of relevant variables of this physically motivated example in adaptive flight control.

The simulation results are shown in Figure \ref{f:Error_2Step_Response}. Both the standard MRAC (\ref{e:MRAC}) and higher order MRAC (\ref{e:Accelerated_MRAC}) (denoted ``A-MRAC'') algorithms are seen to converge in both command tracking and model tracking error to zero. The higher order algorithm however, can be seen to converge at a faster rate. Additionally, the higher order algorithm can be seen to result in fewer oscillations which may be due to the presence of damping in the algorithm and the filtering effect as in (\ref{e:Accelerated_MRAC}). The rapid reduction in oscillations is desirable, particularly given that the system was provided a constant command.

\section{Conclusions and related work}
\label{s:Conclusion}

In this work we derived higher order algorithms for optimization and learning in time-varying and dynamical machine learning problems. The variational approach taken provides a unified method for analyzing both algebraic and dynamical error models, as demonstrated by the regression with time-varying features and adaptive control problems. Our higher order algorithms were proven to be stable, with bounded \emph{constant regret}, thus lending to application in real-time, safety-critical, sequential decision making problems where provably correct algorithms must be employed.

Learning for dynamical systems has been an active area of research within the machine learning community, especially within the area of reinforcement learning \cite{Bertsekas_2017,Sutton_2018,Tu_2018,Recht_2018}. There has also been a large increase in recent work studying learning and control for unknown linear dynamical systems: least squares \cite{Simchowitz_2018}, linear quadratic regulator robust control \cite{Dean_2018,Dean_2018a,Dean2018}, and spectral filtering \cite{Hazan_2017,Hazan_2018a}. One major difference between these works and the one presented here is that our algorithms are streaming and even provide for constant regret for open loop unstable systems.

This work continues in the tradition of \cite{Su_2016} and \cite{Wibisono_2016} whereby insight is gained into higher order gradient descent methods through a continuous lens. Continuous time analysis of machine learning algorithms is becoming increasingly prevalent in training deep neural networks \cite{Arora_2018,Arora_2019} as well as continuous networks \cite{Chen_2018,Arora_2019a}. Future work will be to obtain discrete time implementations of our algorithms \cite{Wilson_2018,Betancourt_2018}, and to connect those back to discrete time adaptive algorithms \cite{Goodwin_1980,Goodwin_1981,Goodwin_1984}.

\clearpage
\bibliography{References}
\bibliographystyle{ieeetr}

\clearpage
\appendix
\noindent{\LARGE \textbf{Appendix}}

\noindent{\textbf{Organization of the appendix}}

Mathematical definitions and Barbalat's lemma are provided in Appendix \ref{s:Definitions}. Lyapunov stability definitions, stability analysis, and regret bounds of the algorithms presented in this paper are provided in Appendix \ref{s:Stability_Analysis}. Appendix \ref{s:Experiment_MRAC_Implementation} provides details regarding the model reference adaptive control simulation implementation for the physically motivated example in flight control, as well as additional plots demonstrating the effects of an increase in the time variation of a feature on the stability of the regression algorithms presented in this paper.

\section{Definitions and Barbalat's Lemma}\label{s:Definitions}
This section details useful definitions regarding signals which are used throughout this paper.

\begin{definition}[See \cite{Narendra2005}]\label{d:L_p}
    For any fixed $p\in[1,\infty)$, $f:\mathbb{R}^+\rightarrow\mathbb{R}$ is defined to belong to $\mathcal{L}_p$ if $f$ is locally integrable and
    \begin{equation*}
        \lVert f(t)\rVert_{\mathcal{L}_p}\triangleq\left(\lim_{t\rightarrow\infty}\int_0^t\lVert f(\tau)\rVert^pd\tau\right)^{\frac{1}{p}}<\infty.
    \end{equation*}
    When $p=\infty$, $f\in\mathcal{L}_{\infty}$, if,
    \begin{equation*}
        \lVert f\rVert_{\mathcal{L}_{\infty}}\triangleq\sup_{t\geq0}\lVert f(t)\rVert<\infty.
    \end{equation*}
\end{definition}
The notion of persistence of excitation has a long history in adaptive systems. It is commonly used to denote the condition that all states in the adaptive system are excited to allow for perfect system identification. Here it refers the the condition of all of the states of a system (collectively the regressor) being excited such that parameter convergence occurs. The following definition is regarding persistence of excitation:
\begin{definition}[See \cite{Jenkins_2018}]\label{d:PE}
    Let $\omega\in[t_0,\infty)\rightarrow\mathbb{R}^p$ be a time-varying parameter with initial condition defined as $\omega_0=\omega(t_0)$; then the parameterized function of time $y(t,\omega):[t_0,\infty)\times\mathbb{R}^p\rightarrow\mathbb{R}^m$ is persistently exciting if there exists $T>0$ and $\alpha>0$ such that
    \begin{equation*}
        \int_t^{t+T}y(\tau,\omega)y^T(\tau,\omega)d\tau\succeq\alpha I
    \end{equation*}
    for all $t\geq t_0$ and $\omega_0\in\mathbb{R}^p$.
\end{definition}
The notation $X\succeq Y$ denotes that $X-Y$ is positive semidefinite for square matrices $X,Y$ of the same dimension.

The following lemma was attributed to Barbalat in \cite{Popov_1973} and has found significant use in the field of adaptive control and nonlinear control. The version from \cite{Narendra2005} is stated below with an associated corollary:
\begin{lemma}[See \cite{Narendra2005}]
    If $f:\mathbb{R}^+\rightarrow\mathbb{R}$ is uniformly continuous for $t\geq0$, and if the limit of the integral
    \begin{equation*}
        \lim_{t\rightarrow\infty}\int_0^t|f(\tau)|d\tau
    \end{equation*}
    exists and is finite, then
    \begin{equation*}
        \lim_{t\rightarrow\infty}f(t)=0.
    \end{equation*}
\end{lemma}
\begin{corollary}\label{c:Barbalat_Corollary}
    If $f\in\mathcal{L}_2\cap\mathcal{L}_{\infty}$, and $\dot{f}\in\mathcal{L}_{\infty}$, then $\lim_{t\rightarrow\infty}f(t)=0.$
\end{corollary}

\clearpage
\section{Stability analysis}\label{s:Stability_Analysis}
An overview of Lyapunov functions and their use in stability analysis is presented in Section \ref{ss:Stability_Lyapunov}. Stability analysis for the first order update law of Section \ref{ss:TV_Reg} is presented in Section \ref{ss:Stability_TV_Reg}. Section \ref{ss:Stability_MRAC} then presents stability analysis of the first order update law of Section \ref{ss:MRAC}. Stability analysis for the Lyapunov function by \cite{Wibisono_2016} in Table \ref{t:Comparison_Lyap} for time-varying regression is presented in Section \ref{ss:Stability_Wibisono}. Section \ref{ss:Stability_Error1} details the proof of stability of the higher order algorithm derived in Section \ref{ss:Accelerated_Error1}. Stability analysis is provided for the higher order update law of Section \ref{ss:Accelerated_Error2} in Section \ref{ss:Stability_Error2}. The connections between Lyapunov stability and constant regret bounds are made in Section \ref{ss:Regret}.

\subsection{Lyapunov functions}\label{ss:Stability_Lyapunov}

This section provides a primer on Lyapunov functions and some of their common uses. While Lyapunov functions are ubiquitous in control theory and many similar definitions exist, this section was adapted from the definitions by \cite{Narendra2005}. Consider a general nonlinear dynamical system of the form:
\begin{equation}\label{e:nonlinear_eq}
    \dot{x}=f(x,t),\quad x(t_0)=x_0
\end{equation}
where $f(0,t)=0$ $\forall t>0$. Lyapunov functions are often used to determine whether the equilibrium state of the dynamical system in (\ref{e:nonlinear_eq}) is stable, without explicitly finding the solution of (\ref{e:nonlinear_eq}). This is due to the potential difficulty in finding a solution of the nonlinear differential equation in (\ref{e:nonlinear_eq}). The method follows from finding a scalar function $V(x,t)$ of the states $x$ of a system and time. The time derivative $\dot{V}(x,t)$ is then analyzed for all trajectories of the system in (\ref{e:nonlinear_eq}). The notion of a Lyapunov function comes from a energy perspective in which energy in a purely dissipative system is always positive and the time derivative is non-positive. It can be noted that even though the results of this paper rely on Lyapunov functions that are autonomous (i.e., $V(x,t)=V(x)$), some of what will be provided is additionally applicable to non-autonomous systems.

The following theorem establishes uniform asymptotic stability of the nonlinear dynamical system in (\ref{e:nonlinear_eq}), with proof available in \cite{Kalman_1960}.
\begin{theorem}[Lyapunov's Direct Method]\label{th:Lyap_Direct}
    The equilibrium state of (\ref{e:nonlinear_eq}) is uniformly asymptotically stable in the large if a scalar function $V(x,t)$ with continuous first partial derivatives with respect to $x$ and $t$ exists such that $V(0,t)=0$ and if the following conditions are satisfied:
    \begin{enumerate}
        \item $V(x,t)$ is positive definite, i.e. there exists a continuous non-decreasing scalar function $\underline{\alpha}$ such that $\underline{\alpha}(0)=0$ and, for all $t$ and all $x\neq0$:
        \begin{equation*}
            0<\underline{\alpha}(\lVert x\rVert)\leq V(x,t)
        \end{equation*}
        \item There exists a continuous non-decreasing scalar function $\underline{\gamma}$ s.t. $\underline{\gamma}(0)=0$ and the derivative $\dot{V}$ of $V$ along all system directions is negative-definite; that is that $\dot{V}$ satisfies for all $t$:
        \begin{equation*}
            \dot{V}=\frac{\partial V}{\partial t}+(\nabla V)^Tf(x,t)\leq-\underline{\gamma}(\lVert x\rVert)<0,\quad \forall x\neq0
        \end{equation*}
        \item $V(x,t)$ is decreascent, that is, there exists a continuous non-decreasing scalar function $\underline{\beta}$, such that $\underline{\beta}(0)=0$ and for all $t$:
        \begin{equation*}
            V(x,t)\leq\underline{\beta}(\lVert x\rVert)
        \end{equation*}
        \item $V(x,t)$ is radially unbounded, that is:
        \begin{equation*}
            \lim_{\lVert x\rVert\rightarrow\infty}\underline{\alpha}(\lVert x\rVert)=\infty
        \end{equation*}
    \end{enumerate}
\end{theorem}
It should be noted that in general, all of the conditions of Theorem \ref{th:Lyap_Direct} may not hold, in particular the condition 2 of Theorem \ref{th:Lyap_Direct}, which requires $\dot{V}$ being negative-definite along all system directions may be difficult to satisfy. In particular, $\dot{V}<0$ may never hold for the entire state space of adaptive systems as the unknown parameter would have to show up in the expression for $\dot{V}<0$. The parameter being unknown would restrict this condition from holding. However, it is common that $\dot{V}\leq0$, that is, that the time derivative of $V$ is negative semi-definite. The following proposition is used throughout this paper:
\begin{prop}
    If $V(x,t)$ in Theorem \ref{th:Lyap_Direct} is positive definite (condition 1) and $\dot{V}(x,t)\leq0$ then the origin of (\ref{e:nonlinear_eq}) is stable; if in addition, condition 3 of Theorem \ref{th:Lyap_Direct} is satisfied, then uniform stability follows, $x$ is bounded for all time, and $V(x,t)$ is called a Lyapunov function.
\end{prop}
Linear time invariant (LTI) systems are frequently considered in this paper. The following theorem establishes stability for LTI systems and gives a connection to what is known as the Lyapunov equation:
\begin{theorem}[See \cite{Narendra2005}]
    The equilibrium state $x=0$ of the linear time invariant system
    \begin{equation}\label{e:Ax}
        \dot{x}=Ax
    \end{equation}
    is asymptotically stable if, and only if, given any symmetric positive-definite matrix $Q$, there exists a symmetric positive-definite matrix $P$, which is the unique solution of the set of $n(n+1)/2$ linear equations (called the Lyapunov equation):
    \begin{equation}\label{e:Lyapunov_Equation}
        A^TP+PA=-Q.
    \end{equation}
    Therefore, $V(x)=x^TPx$ is a Lyapunov function for equation (\ref{e:Ax}).
\end{theorem}

\subsection{Stability analysis of the first order time-varying regression algorithm}\label{ss:Stability_TV_Reg}

\begin{proof}[Proof of stability of the first order update in (\ref{e:update_GF}) for the regression error model in (\ref{e:error1})]
~\\
Consider the following Lyapunov function candidate:
\begin{equation}\label{e:Lyap_error1}
V(\tilde{\theta}(t))=\frac{1}{2\gamma}\tilde{\theta}^T(t)\tilde{\theta}(t)
\end{equation}
which is a non-negative scalar quantity. The time derivative of this Lyapunov function candidate is:
\begin{equation*}
\dot{V}(\tilde{\theta}(t))=\frac{1}{\gamma}\tilde{\theta}^T(t)\dot{\theta}(t).
\end{equation*}
Employing the equation for output error (\ref{e:error1}), as well as the first order algorithm (\ref{e:update_GF}), the time derivative of the Lyapunov function may be expressed as:
\begin{equation*}
\dot{V}(\tilde{\theta}(t))=-e_y^2(t)\leq 0.
\end{equation*}
From this, it can be concluded that $V(\tilde{\theta})$ is a Lyapunov function and $\tilde{\theta}\in\mathcal{L}_{\infty}$. By integrating $\dot{V}$ from $t_0$ to $\infty$: $\int_{t_0}^{\infty}e_y^2(t)dt=-\int_{t_0}^{\infty}\dot{V}dt=V(\tilde{\theta}(t_0))-V(\tilde{\theta}(\infty))<\infty$, thus $e_y\in\mathcal{L}_2$. If in addition $\phi\in\mathcal{L}_{\infty}$ (the magnitude of the features are bounded), then from equation (\ref{e:error1}) it can be seen that $e_y\in\mathcal{L}_2\cap\mathcal{L}_{\infty}$ and from equation (\ref{e:error1}) and (\ref{e:update_GF}), $\dot{\tilde{\theta}}\in\mathcal{L}_{\infty}$. Also from (\ref{e:update_GF}), given that $e_y\in\mathcal{L}_2$, it can be seen that $\dot{\tilde{\theta}}\in\mathcal{L}_2\cap\mathcal{L}_{\infty}$. If the additional assumption is made that $\dot{\phi}\in\mathcal{L}_{\infty}$ (the time derivative of the features are bounded), then from (\ref{e:error_model1_error}), $\dot{e}_y\in\mathcal{L}_{\infty}$. Additionally from (\ref{e:update_GF}), it can then be seen that $\ddot{\tilde{\theta}}\in\mathcal{L}_{\infty}$. Then from Corollary \ref{c:Barbalat_Corollary} in Appendix \ref{s:Definitions}:
\begin{equation*}
    \lim_{t\rightarrow\infty}e_y(t)=0\quad\text{and}\quad\lim_{t\rightarrow\infty}\dot{\tilde{\theta}}(t)=0
\end{equation*}
which is to say that the estimation error goes to zero as time goes to infinity, and the parameter estimate reaches a constant steady state value.
\end{proof}
For the parameter estimation error $\tilde{\theta}\rightarrow0$, persistence of excitation is needed (see Appendix \ref{s:Definitions}, Definition \ref{d:PE}). This condition is similar in machine learning problems, where the objective function is defined based on a estimation error, but parameter convergence is not guaranteed without sufficient richness of data.

\subsection{Stability analysis of the first order adaptive control and identification algorithm}\label{ss:Stability_MRAC}

\begin{proof}[Proof of stability of the first order update in (\ref{e:MRAC}) for the adaptive control error model in (\ref{e:error_model_2})]
~\\
Consider the following Lyapunov function candidate:
\begin{equation}
    V(e(t),\tilde{\theta}(t))=e^T(t)Pe(t)+\frac{1}{\gamma}\tilde{\theta}^T(t)\tilde{\theta}(t).
\end{equation}
The time derivative of the Lyapunov function candidate may be expressed as:
\begin{equation*}
    \dot{V}(e(t),\tilde{\theta}(t))=2e^T(t)P\dot{e}(t)+\frac{2}{\gamma}\tilde{\theta}^T(t)\dot{\theta}(t).
\end{equation*}
Employing the Lyapunov equation (\ref{e:Lyapunov_Equation}), the equation for model tracking error model (\ref{e:error_model_2}), as well as the first order algorithm (\ref{e:MRAC}), the time derivative of the Lyapunov function may be expressed as:
\begin{equation*}
    \dot{V}(e(t),\tilde{\theta}(t))=-e^T(t)Qe(t)\leq0.
\end{equation*}
Thus it can be concluded that $V(e(t),\tilde{\theta}(t))$ is a Lyapunov function with $e\in\mathcal{L}_{\infty}$ and $\tilde{\theta}\in\mathcal{L}_{\infty}$. By integrating $\dot{V}$ from $t_0$ to $\infty$: $\int_{t_0}^{\infty}e^T(t)Qe(t)dt=-\int_{t_0}^{\infty}\dot{V}dt=V(t_0)-V(\infty)<\infty$, thus $e\in\mathcal{L}_2$. If in addition $\phi\in\mathcal{L}_{\infty}$\footnote{As is common in adaptive control, $\phi=x$. It was proved that $e\in\mathcal{L}_{\infty}$, with $\hat{x}\in\mathcal{L}_{\infty}$ by design of a suitable input $u$. Thus with $x=\hat{x}-e$, $\phi=x$ is bounded by construction and thus this is not a restrictive assumption.}, then from equation (\ref{e:error_model_2}) $\dot{e}\in\mathcal{L}_{\infty}$ and from Corollary \ref{c:Barbalat_Corollary} in Appendix \ref{s:Definitions}:
\begin{equation}
    \lim_{t\rightarrow\infty}e(t)=0
\end{equation}
which is to say that the model tracking error goes to zero as time goes to infinity.
\end{proof}
Once again, for the parameter estimation error $\tilde{\theta}\rightarrow0$, persistence of excitation of the regressor of the system is needed (see Appendix \ref{s:Definitions}, Definition \ref{d:PE}). It can be noted that compared to the stability analysis for time-varying regression in Appendix \ref{ss:Stability_TV_Reg}, $\lim_{t\rightarrow\infty}e(t)=0$ when $\phi\in\mathcal{L}_{\infty}$ without the additional requirement that $\dot{\phi}\in\mathcal{L}_{\infty}$.

\subsection{Stability using the Lyapunov function in \cite{Wibisono_2016} for time-varying regression}\label{ss:Stability_Wibisono}

The candidate Lyapunov function proposed for stability in \cite{Wibisono_2016} Equation 8 is restated as:
\begin{equation*}
    V=D_h(\theta^*,\theta+\text{e}^{-\bar{\alpha}_t}\dot{\theta})+\text{e}^{\bar{\beta}_t}(L(\theta)-L(\theta^*))
\end{equation*}
where the Bregman divergence ($D_h(y,x)=h(y)-h(x)-\left<\nabla h(x),y-x\right>$) may be expanded with the same squared Euclidean norm ($h(x)=\frac{1}{2}\lVert x\rVert^2$) and squared loss ($L=\frac{1}{2}e_y^2$) considered in Section \ref{ss:Accelerated_Error1} as:
\begin{equation*}
    V=\frac{1}{2}\left\lVert\tilde{\theta}+\frac{1}{\beta(1+\mu\phi^T\phi)}\dot{\theta}\right\rVert^2+\frac{\gamma}{\beta(1+\mu\phi^T\phi)}\frac{1}{2}e_y^2.
\end{equation*}
Evaluating the time derivative of this candidate Lyapunov function using the time-varying regression error model (\ref{e:error1}), its time derivative (\ref{e:error_model1_error}) and higher order algorithm (\ref{e:Accelerated_GF_2nd}):
\begin{equation*}
    \dot{V}=-\gamma e_y^2\left(1+\frac{\mu\phi^T{\color{red}\dot{\phi}}}{\beta(1+\mu\phi^T\phi)^2}\right)+\frac{\gamma}{\beta(1+\mu\phi^T\phi)}e_y\tilde{\theta}^T{\color{red}\dot{\phi}}
\end{equation*}
which can be seen to be sign indeterminate. There can exist time derivatives of the feature ${\color{red}\dot{\phi}}$ for which $\dot{V}$ is positive and thus global stability cannot be established for arbitrary feature time variations. It can be noted that if the feature is constant, as is assumed implicitly by \cite{Wibisono_2016} (i.e., ${\color{red}\dot{\phi}}=0$), then stability can be established.

\subsection{Stability analysis of the higher order time-varying regression algorithm in Theorem \ref{th:Stability_Error1}}\label{ss:Stability_Error1}

\begin{proof}[Proof of Theorem \ref{th:Stability_Error1}]
It can be noted that the Lyapunov function proposed in \cite{Wibisono_2016} cannot be used to demonstrate stability of the accelerated algorithm in (\ref{e:Accelerated_GF}), as shown in Appendix \ref{ss:Stability_Wibisono}. To show stability for the accelerated update law (\ref{e:Accelerated_GF}) for time-varying regression (\ref{e:error1}), consider the following candidate Lyapunov function inspired by the higher order tuner approach in \cite{Evesque_2003}:
\begin{equation}\label{e:V_Accelerated_GF}
V=\frac{1}{\gamma}\lVert\vartheta-\theta^*\rVert^2+\frac{1}{\gamma}\lVert\theta-\vartheta\rVert^2
\end{equation}
which is a non-negative scalar quantity which represents squared error present in the algorithm. Choosing the normalization parameter\footnote{\label{fn:WLOG}Can be chosen without loss of generality.} in (\ref{e:N_t}) as $\mu=2\gamma/\beta$ and using equations (\ref{e:error1}) and (\ref{e:Accelerated_GF}), the time derivative of the candidate Lyapunov function in (\ref{e:V_Accelerated_GF}) may be bounded as:
\begin{equation*}
    \dot{V}\leq-\frac{2\beta}{\gamma}\lVert\theta-\vartheta\rVert^2-\lVert e_y\rVert^2-\left[\lVert e_y\rVert-2\lVert\theta-\vartheta\rVert\lVert\phi\rVert\right]^2
\end{equation*}
Thus it can be concluded that $V$ is a Lyapunov function with $(\vartheta-\theta^*)\in\mathcal{L}_{\infty}$ and $(\theta-\vartheta)\in\mathcal{L}_{\infty}$. By integrating $\dot{V}$ from $t_0$ to $\infty$: $\int_{t_0}^{\infty}\lVert e_y\rVert^2dt\leq-\int_{t_0}^{\infty}\dot{V}dt=V(t_0)-V(\infty)<\infty$, thus $e_y\in\mathcal{L}_2$. Likewise, $\int_{t_0}^{\infty}\frac{2\beta}{\gamma}\lVert\theta-\vartheta\rVert^2dt\leq-\int_{t_0}^{\infty}\dot{V}dt=V(t_0)-V(\infty)<\infty$, thus $(\theta-\vartheta)\in\mathcal{L}_2\cap\mathcal{L}_{\infty}$. Furthermore:
\begin{equation*}
    \lVert\theta-\vartheta\rVert^2_{\mathcal{L}_2}\leq\frac{\gamma V(t_0)}{2\beta}
\end{equation*}
Here the effect of the parameter $\beta$ is very apparent once again. As $\beta\rightarrow\infty$, $\lVert\theta-\vartheta\rVert^2_{\mathcal{L}_2}\rightarrow0$. If in addition $\phi\in\mathcal{L}_{\infty}$ (the magnitude of the features are bounded), then from equation (\ref{e:error1}) $e_y\in\mathcal{L}_2\cap\mathcal{L}_{\infty}$, and from equation (\ref{e:Accelerated_GF}) $\dot{\vartheta},\dot{\tilde{\theta}}\in\mathcal{L}_2\cap\mathcal{L}_{\infty}$. If the additional assumption is made that $\dot{\phi}\in\mathcal{L}_{\infty}$ (the time derivative of the features are bounded), then from equation (\ref{e:error_model1_error}), it can be seen that $\dot{e}_y\in\mathcal{L}_{\infty}$ and from equation (\ref{e:Accelerated_GF}) $\ddot{\vartheta},\ddot{\tilde{\theta}}\in\mathcal{L}_{\infty}$ and thus from Corollary \ref{c:Barbalat_Corollary} in Appendix \ref{s:Definitions}:
\begin{equation*}
    \lim_{t\rightarrow\infty}e_y(t)=0,~\lim_{t\rightarrow\infty}(\theta(t)-\vartheta(t))=0,~\lim_{t\rightarrow\infty}\dot{\vartheta}(t)=0,~\lim_{t\rightarrow\infty}\dot{\tilde{\theta}}(t)=0
\end{equation*}
which is to say that the estimation error goes to zero as time goes to infinity, and the parameter estimate and algorithm reach a steady state value.
\end{proof}
For the parameter estimation error to converge to zero ($\tilde{\theta}\rightarrow0$), persistence of excitation of the system regressor is needed (see Appendix \ref{s:Definitions}, Definition \ref{d:PE}).

\subsection{Stability analysis of the higher order adaptive control and identification algorithm in Theorem \ref{th:Stability_Error2}}\label{ss:Stability_Error2}

\begin{proof}[Proof of Theorem \ref{th:Stability_Error2}]
To show stability for the accelerated update law (\ref{e:Accelerated_MRAC}) for the dynamical error model in (\ref{e:error_model_2}), consider the following candidate Lyapunov function inspired by the higher order tuner approach for adaptive control in \cite{Evesque_2003}:
\begin{equation}\label{e:V_Accelerated_MRAC}
V=\frac{1}{\gamma}\lVert\vartheta-\theta^*\rVert^2+\frac{1}{\gamma}\lVert\theta-\vartheta\rVert^2+e^TPe
\end{equation}
which can be seen to be (\ref{e:V_Accelerated_GF}) with an additional term corresponding to the model tracking error. Choosing the normalization parameter\footnoteref{fn:WLOG} in (\ref{e:N_t}) as $\mu=2\gamma\lVert Pb\rVert^2/\beta$ and the symmetric positive definite matrix\footnoteref{fn:WLOG} in the Lyapunov equation ($A^TP+PA=-Q$) from before as $Q=2I$ and using equations (\ref{e:error_model_2}) and (\ref{e:Accelerated_MRAC}), the time derivative of the Lyapunov function in (\ref{e:V_Accelerated_MRAC}) may be bounded as:
\begin{equation*}
    \dot{V}\leq-\frac{2\beta}{\gamma}\lVert\theta-\vartheta\rVert^2-\lVert e\rVert^2-\left[\lVert e\rVert-2\lVert Pb\rVert\lVert\theta-\vartheta\rVert\lVert\phi\rVert\right]^2
\end{equation*}
Thus it can be concluded that $V$ is a Lyapunov function with $e\in\mathcal{L}_{\infty}$, $(\vartheta-\theta^*)\in\mathcal{L}_{\infty}$, and $(\theta-\vartheta)\in\mathcal{L}_{\infty}$. By integrating $\dot{V}$ from $t_0$ to $\infty$: $\int_{t_0}^{\infty}\lVert e\rVert^2dt\leq-\int_{t_0}^{\infty}\dot{V}dt=V(t_0)-V(\infty)<\infty$, thus $e\in\mathcal{L}_2\cap\mathcal{L}_{\infty}$. Likewise, $\int_{t_0}^{\infty}\frac{2\beta}{\gamma}\lVert\theta-\vartheta\rVert^2dt\leq-\int_{t_0}^{\infty}\dot{V}dt=V(t_0)-V(\infty)<\infty$, thus $(\theta-\vartheta)\in\mathcal{L}_2\cap\mathcal{L}_{\infty}$. Furthermore, it can be concluded that again: $\lVert\theta-\vartheta\rVert^2_{\mathcal{L}_2}\leq\frac{\gamma V(t_0)}{2\beta}$, where as $\beta\rightarrow\infty$: $\lVert\theta-\vartheta\rVert^2_{\mathcal{L}_2}\rightarrow0$. If in addition $\phi\in\mathcal{L}_{\infty}$\footnote{As is common in adaptive control, $\phi=x$. It was proved that $e\in\mathcal{L}_{\infty}$, with $\hat{x}\in\mathcal{L}_{\infty}$ by design of a suitable input $u$. Thus with $x=\hat{x}-e$, $\phi=x$ is bounded by construction and thus this is not a restrictive assumption.}, then from equation (\ref{e:error_model_2}) $\dot{e}\in\mathcal{L}_{\infty}$, and thus from Corollary \ref{c:Barbalat_Corollary}:
\begin{equation*}
    \lim_{t\rightarrow\infty}e(t)=0
\end{equation*}
which is to say that the model tracking error goes to zero as time goes to infinity. It can be noted that compared to the stability analysis in Section \ref{ss:Stability_Error1}, $\lim_{t\rightarrow\infty}e(t)=0$ when $\phi\in\mathcal{L}_{\infty}$ without the additional requirement that $\dot{\phi}\in\mathcal{L}_{\infty}$. Also, from equation (\ref{e:Accelerated_MRAC}) $\dot{\vartheta},\dot{\tilde{\theta}}\in\mathcal{L}_2\cap\mathcal{L}_{\infty}$. If the additional assumption is made that $\dot{\phi}\in\mathcal{L}_{\infty}$\footnote{This is not a restrictive assumption in adaptive control with $\phi=x$, as $\dot{x}=\dot{\hat{x}}-\dot{e}$ is bounded by construction.}, then from equation (\ref{e:Accelerated_MRAC}) $\ddot{\vartheta},\ddot{\tilde{\theta}}\in\mathcal{L}_{\infty}$, and thus from Corollary \ref{c:Barbalat_Corollary}:
\begin{equation*}
    \lim_{t\rightarrow\infty}(\theta(t)-\vartheta(t))=0,~\lim_{t\rightarrow\infty}\dot{\vartheta}(t)=0,~\lim_{t\rightarrow\infty}\dot{\tilde{\theta}}(t)=0
\end{equation*}
which states that the parameter estimate and algorithm reach a steady state value.
\end{proof}
For the parameter estimation error $\tilde{\theta}\rightarrow0$, persistence of excitation of the regressor of the system is needed (see Appendix \ref{s:Definitions}, Definition \ref{d:PE}).

\subsection{Constant regret and Lyapunov stability}\label{ss:Regret}

The efficiency of an algorithm in online optimization in machine learning is often analyzed using the notion of ``regret'' in discrete time as
\begin{equation}\label{e:regret}
    \text{Regret}=\sum_{k=1}^T\mathcal{C}_k(\theta_k)-\min_{\theta\in\Theta}\sum_{k=1}^T\mathcal{C}_k(\theta)
\end{equation}
where $k\in\mathbb{N}$ is the time index and $\Theta$ is a compact convex set where the parameters reside. Regret corresponds to the sum of the time-varying convex costs $\mathcal{C}_k$ associated with the choice of the time-varying parameter estimate $\theta_k$, minus the cost associated with the best static parameter estimate choice in hindsight, over a time horizon of $T$ steps \cite{Zinkevich_2003,Hazan_2007,Hazan_2008,Shalev_Shwartz_2011,Hazan_2016}.
\begin{proof}[Proof of Corollaries \ref{c:Stability_Error1} and \ref{c:Stability_Error2}]
Suppose we consider squared output error (respectively squared model tracking error) cost, consistent with the squared loss employed in this paper: $\mathcal{C}_k=\lVert e_{y,k}\rVert^2$. A continuous time limit of (\ref{e:regret}) leads to an integral as
\begin{equation}\label{e:regret_continuous}
    \text{Regret}_\mathrm{continuous}:=\int_{t_0}^T \lVert e_y(\tau)\rVert^2 d\tau
\end{equation}
where for time varying regression $\mathcal{C}_k(\theta^*)=0$ as seen in equation (\ref{e:error1}), and an exponentially decaying term due to initial conditions may be present for the dynamical error model in (\ref{e:error_model_2}) \cite{Narendra2005}. Continuous regret can be connected to Lyapunov stability, given that $V(t)>0$ and $\dot{V}(t)\leq-\lVert e_y(t) \rVert^2\leq0$, $\dot{V}(t)\leq-\lVert e(t) \rVert^2\leq0$, as demonstrated in Appendices \ref{ss:Stability_Error1} and \ref{ss:Stability_Error2}. By integrating $\dot{V}$ from $t_0$ to $T$, we obtain
\begin{equation}\label{e:regret_AC}
    \int_{t_0}^{T}\lVert e_y(\tau)\rVert^2d\tau\leq-\int_{t_0}^{T}\dot{V}(\tau)d\tau=V(t_0)-V(T).
\end{equation}
where $e(\tau)$ may be employed in (\ref{e:regret_continuous}) and (\ref{e:regret_AC}) for the dynamical error model in (\ref{e:error_model_2}). Given that $\dot{V}(t)\leq0$, it can be seen that $V(t_0)-V(T)\leq V(t_0)=\mathcal{O}(1)$.
\end{proof}
A close connection can thus be seen between continuous regret in (\ref{e:regret_continuous}) and Lyapunov stability in (\ref{e:regret_AC}). As stated in the field of online optimization, it is desired to have regret grow sub-linearly with time, such that average regret, $(1/T)\text{Regret}$, goes to zero in the limit $T\rightarrow\infty$. Such an algorithm is stated to be an efficient algorithm \cite{Hazan_2016}. By employing Lyapunov stability theory from the field of adaptive control, we have shown convergence of output/state errors to zero for our algorithms with an integral which is akin to \emph{constant} regret upper bounded by $V(t_0)$ in (\ref{e:regret_AC}). Thus our regret bound does not increase as a function of time as is common in online machine learning approaches \cite{Zinkevich_2003,Hazan_2007,Hazan_2008,Shalev_Shwartz_2011,Hazan_2016}. Regret contains a sum of non-negative costs and is therefore a non-decreasing function of the time horizon $T$. Thus $O(1)$, constant regret attained by our algorithms is the best achievable regret, up to constants which do not vary with time.

\clearpage
\section{Implementation details and additional simulation plots}
\label{s:Experiment_MRAC_Implementation}

This section provides implementation details for the state feedback adaptive control simulation in Section \ref{ss:Error_2_Experiments}. The F-16 model used in this paper is from \cite{Stevens2003}.\footnote{Model downloaded from: \url{http://www.aem.umn.edu/~balas/darpa_sec/SEC.Software.html}.} A trim point for this nonlinear F-16 vehicle model was obtained at a straight and level flying condition at a velocity of $500$ ft/s with an altitude of $15,000$ ft. The model was linearized about this trim point in order to obtain linear dynamics for control design and simulation. The short period linearized longitudinal dynamics of the aircraft are considered in this paper, as is typical for inner loop flight control \cite{Lavretsky2013}. The longitudinal short period variables are:
\begin{equation*}
x_p=
\begin{bmatrix}
\alpha&q
\end{bmatrix}^T,
\qquad
u=\delta_e,
\qquad
z_p=q
\end{equation*}
where the longitudinal state $x_p$ is composed of the vehicle's angle of attack $\alpha$ (degrees) and pitch rate $q$ (degrees per second). The pitch rate is a regulated variable $z_p$. The elevator deflection $\delta_e$ (degrees) is an input to the dynamics. The linearized dynamics and input matrices are:
\begin{equation*}
    A_p=
    \begin{bmatrix}
    -0.6398&0.9378\\
    -1.5679&-0.8791
    \end{bmatrix},
    \qquad
    b_p=
    \begin{bmatrix}
    -0.0777\\
    -6.5121
    \end{bmatrix}
\end{equation*}
The goal is to design the control input $u$ so that $z_p$ tracks a bounded command $z_{cmd}$ with zero error. To ensure a zero tracking error, an integral error $x_e$ state is generated as:
\begin{equation*}
\label{e:integral_error}
\dot{x}_e(t)=z_p(t)-z_{cmd}(t).
\end{equation*}
where the integral error state in this paper represents the integral of the pitch rate command tracking error. The complete plant model augments the plant dynamics with the integral of the tracking error and is written as:
\begin{equation*}\label{e:full_plant}
\underbrace{
\begin{bmatrix}
\dot{x}_p(t)\\
\dot{x}_e(t)
\end{bmatrix}
}_{\dot{x}(t)}
=
\underbrace{
\begin{bmatrix}
A_p&0_{2\times1}\\
[0~1]&0\\
\end{bmatrix}
}_{A}
\underbrace{
\begin{bmatrix}
x_p(t)\\
x_e(t)
\end{bmatrix}
}_{x(t)}
+
\underbrace{
\begin{bmatrix}
b_p\\
0
\end{bmatrix}
}_{b}
u(t)
+
\underbrace{
\begin{bmatrix}
0_{2\times1}\\
-1\\
\end{bmatrix}
}_{b_z}
z_{cmd}(t)
\end{equation*}
This can be expressed more compactly as: $\dot{x}(t)=Ax(t)+bu(t)+b_zz_{cmd}(t)$, where $A\in\mathbb{R}^{3\times 3}$, $b\in\mathbb{R}^{3\times 1}$, $b_z\in\mathbb{R}^{3\times 1}$ are the known matrices provided above. A state feedback gain $\theta^*$ may be designed with linear quadratic regulator (LQR) methods in order to stabilize this system. The following cost matrices were employed to penalize the integral command tracking state and the control input:
\begin{equation*}
    Q_{LQR}=
    \begin{bmatrix}
    0&0&0\\
    0&0&0\\
    0&0&1
    \end{bmatrix},
    \qquad
    R_{LQR}=1
\end{equation*}
The Matlab command $\theta^*=lqr(A,b,Q_{LQR},R_{LQR})'$ resulted in the following gain:
\begin{equation*}
    \theta^*=
    \begin{bmatrix}
    0.1965 & -0.3835 & -1.0000
    \end{bmatrix}^T
\end{equation*}
A stable closed loop matrix $A_m$ may then be formulated as:
\begin{equation*}
    A_m\triangleq A-b\theta^{*T}
\end{equation*}
The plant model may then be expressed in a similar manner as Section \ref{ss:MRAC} with the closed loop matrix as:
\begin{equation*}
    \dot{x}(t)=A_mx(t)+b(u(t)+\theta^{*T}x(t))+b_zz_{zmd}(t)
\end{equation*}
A set of desired dynamics, known as the reference model may then be stated with the closed loop matrix as:
\begin{equation*}
    \label{e:Reference_Model}
    \dot{\hat{x}}(t)=A_m\hat{x}(t)+b\left(u+\theta^T(t)x(t)\right)+b_zz_{cmd}(t)
\end{equation*}
In order to track the reference model in an adaptive control formulation, the control input is set as:
\begin{equation*}
    u(t)=-\theta^T(t)x(t)
\end{equation*}
where the adaptive parameter $\theta$ may be adjusted according to the nominal MRAC (\ref{e:MRAC}) and higher order MRAC (\ref{e:Accelerated_MRAC}) update laws. The model tracking error may be stated as $e=\hat{x}-x$. The error model may then be stated as:
\begin{equation*}
    \dot{e}(t)=A_me(t)+b\tilde{\theta}^T(t)x(t)
\end{equation*}
where $\tilde{\theta}=\theta-\theta^*$, and can be seen to have a similar representation to the dynamical error model in equation (\ref{e:error_model_2}), with $\phi=x$.

\clearpage
\begin{figure}
    \centering
    \begin{subfigure}[b]{\textwidth}
        \centerline{
	    \includegraphics[trim={0.25cm 3.5cm 1.5cm 4.4cm},clip,width=0.2\textwidth]{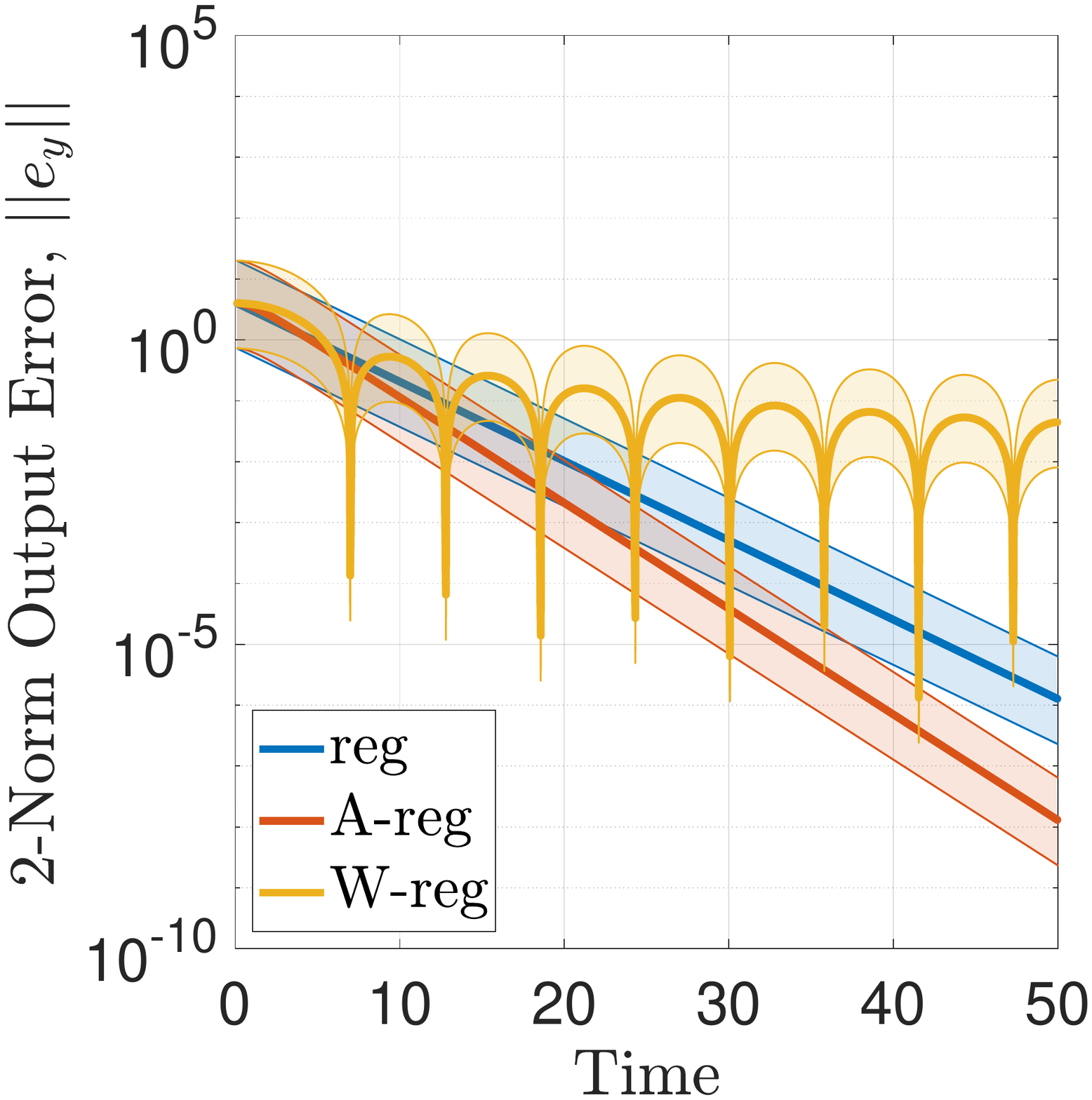}
	    \includegraphics[trim={0.25cm 3.5cm 1.5cm 4.4cm},clip,width=0.2\textwidth]{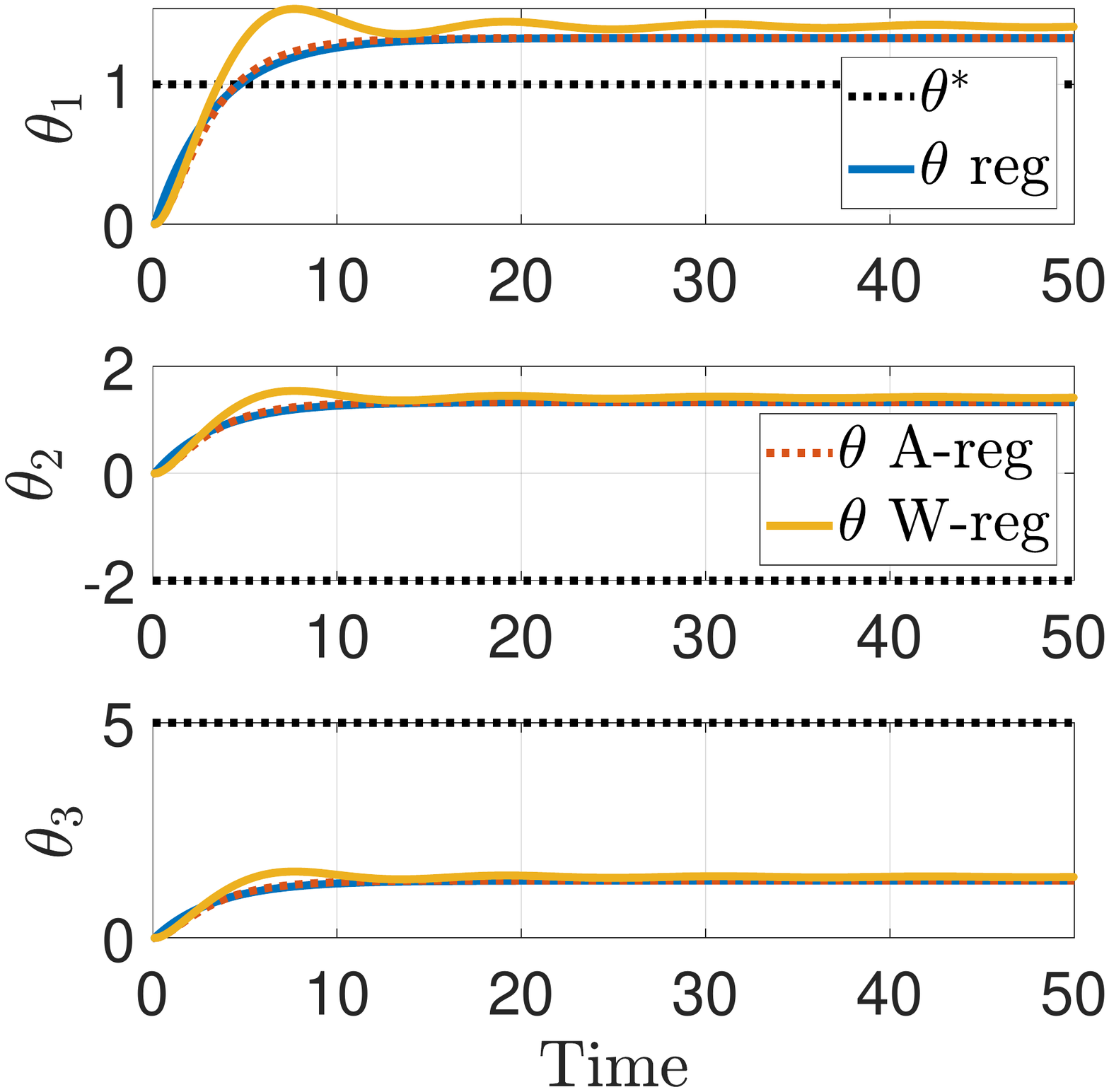}
	    \includegraphics[trim={0.25cm 3.5cm 1.5cm 4.4cm},clip,width=0.2\textwidth]{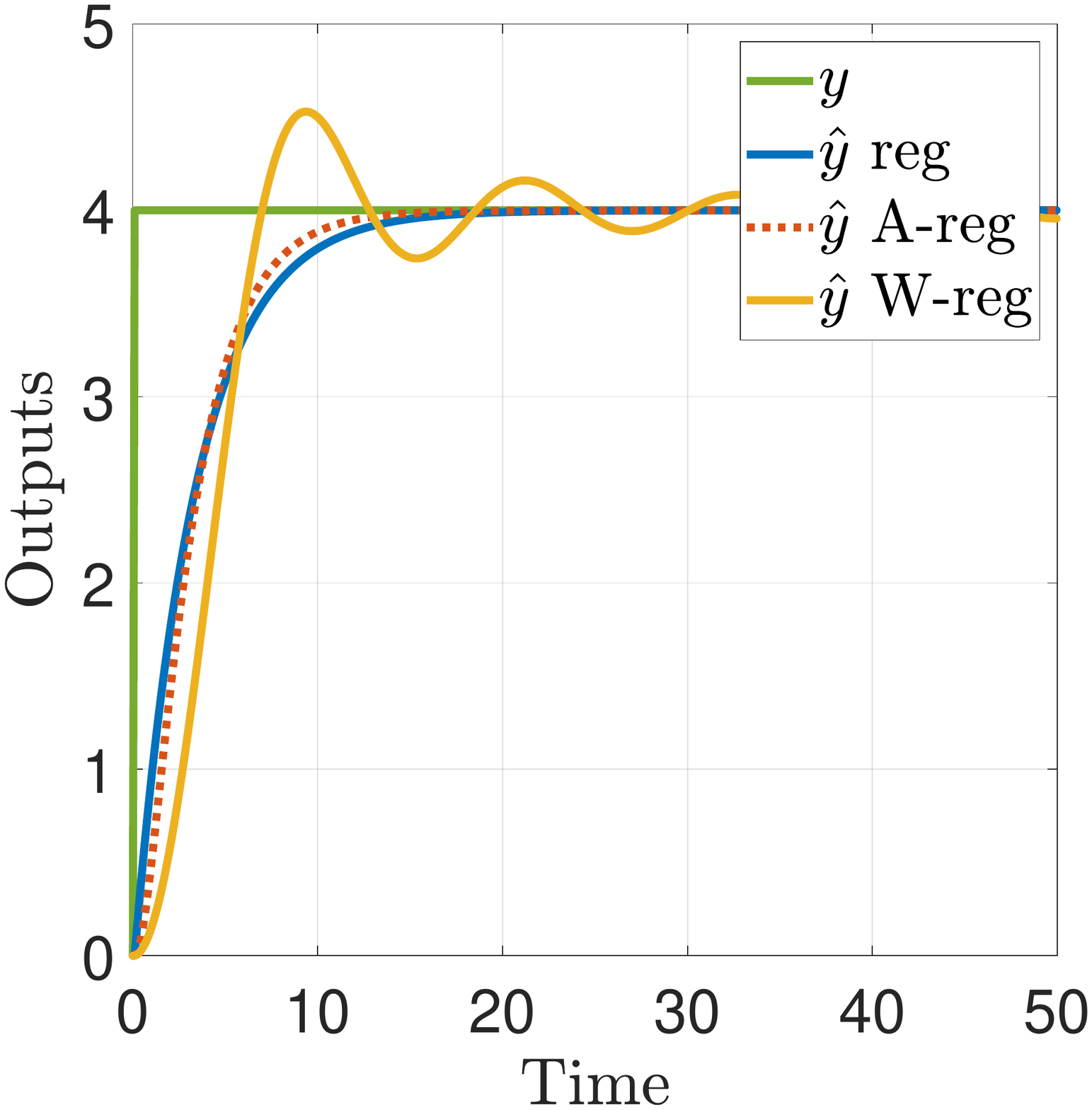}
	    \includegraphics[trim={0.25cm 3.5cm 1.5cm 4.4cm},clip,width=0.2\textwidth]{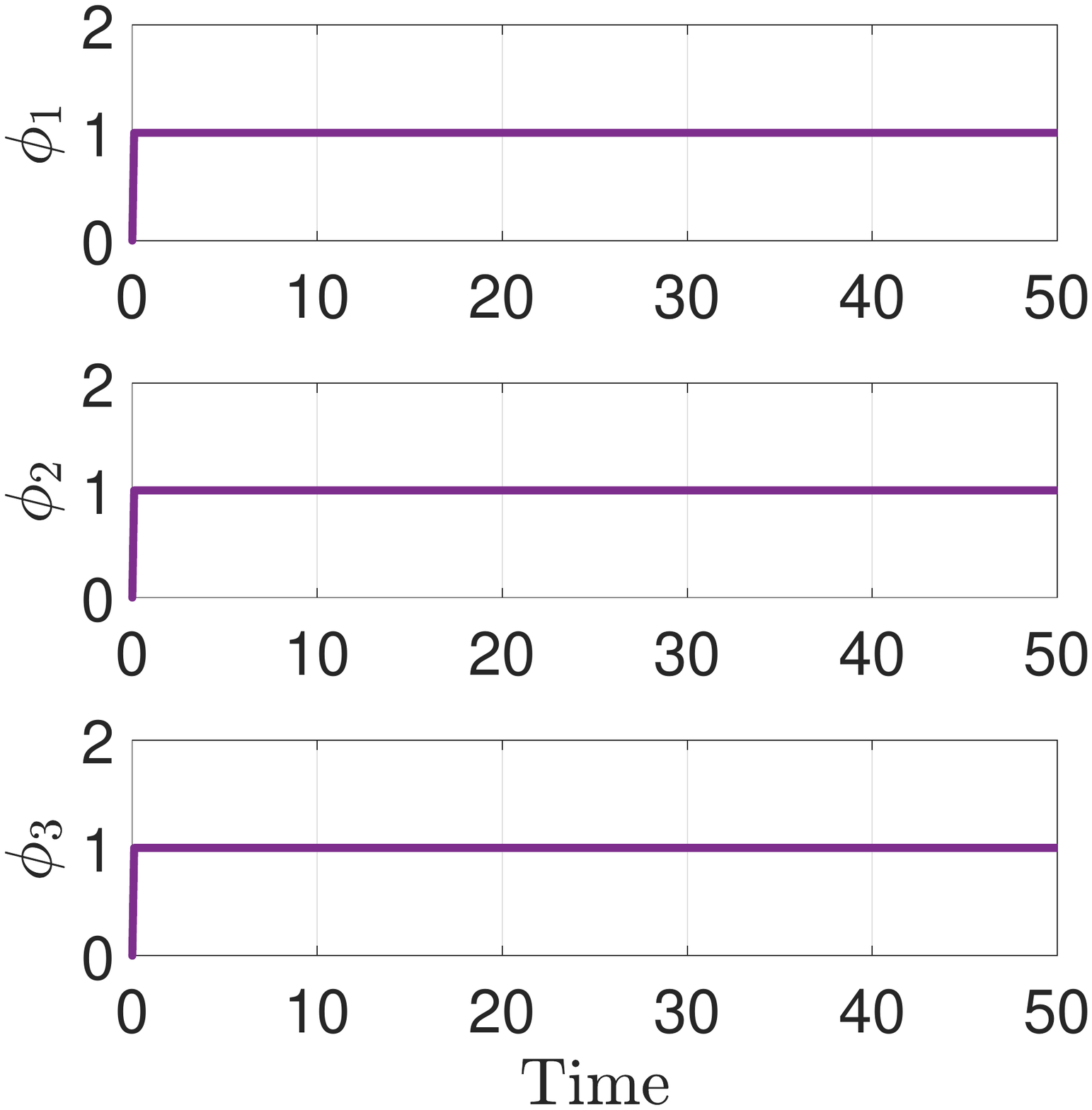}
	    }
	\caption{Time-varying regression: At time $t=0.1$, the feature vector steps to a constant value of $\phi=[1,~1,~1]^T$.}
	\label{f:Error1_Step_Response_Appendix}
    \end{subfigure}
    
    \begin{subfigure}[b]{\textwidth}
        \centerline{
	    \includegraphics[trim={0.25cm 3.5cm 1.5cm 4.4cm},clip,width=0.2\textwidth]{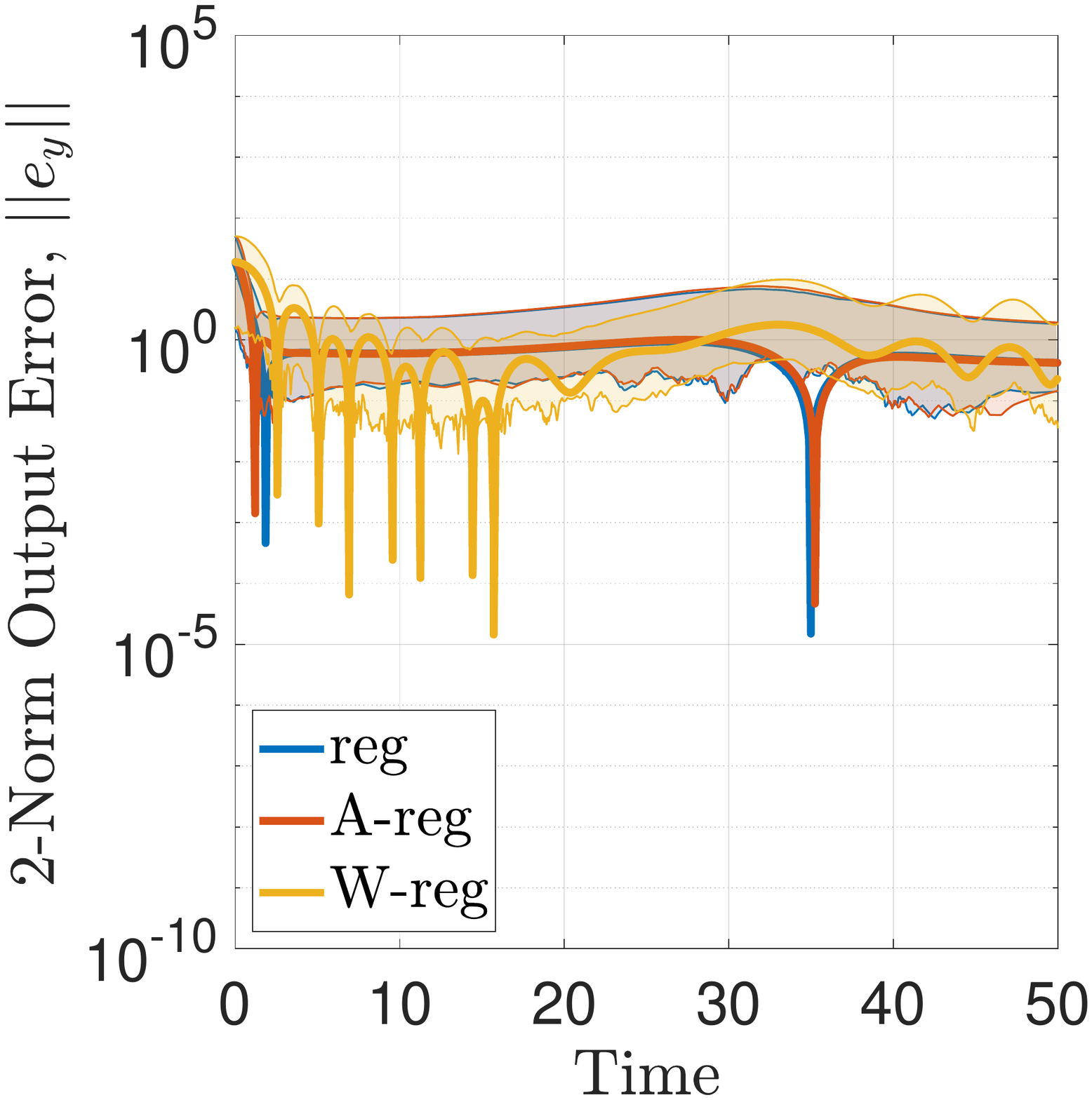}
	    \includegraphics[trim={0.25cm 3.5cm 1.5cm 4.4cm},clip,width=0.2\textwidth]{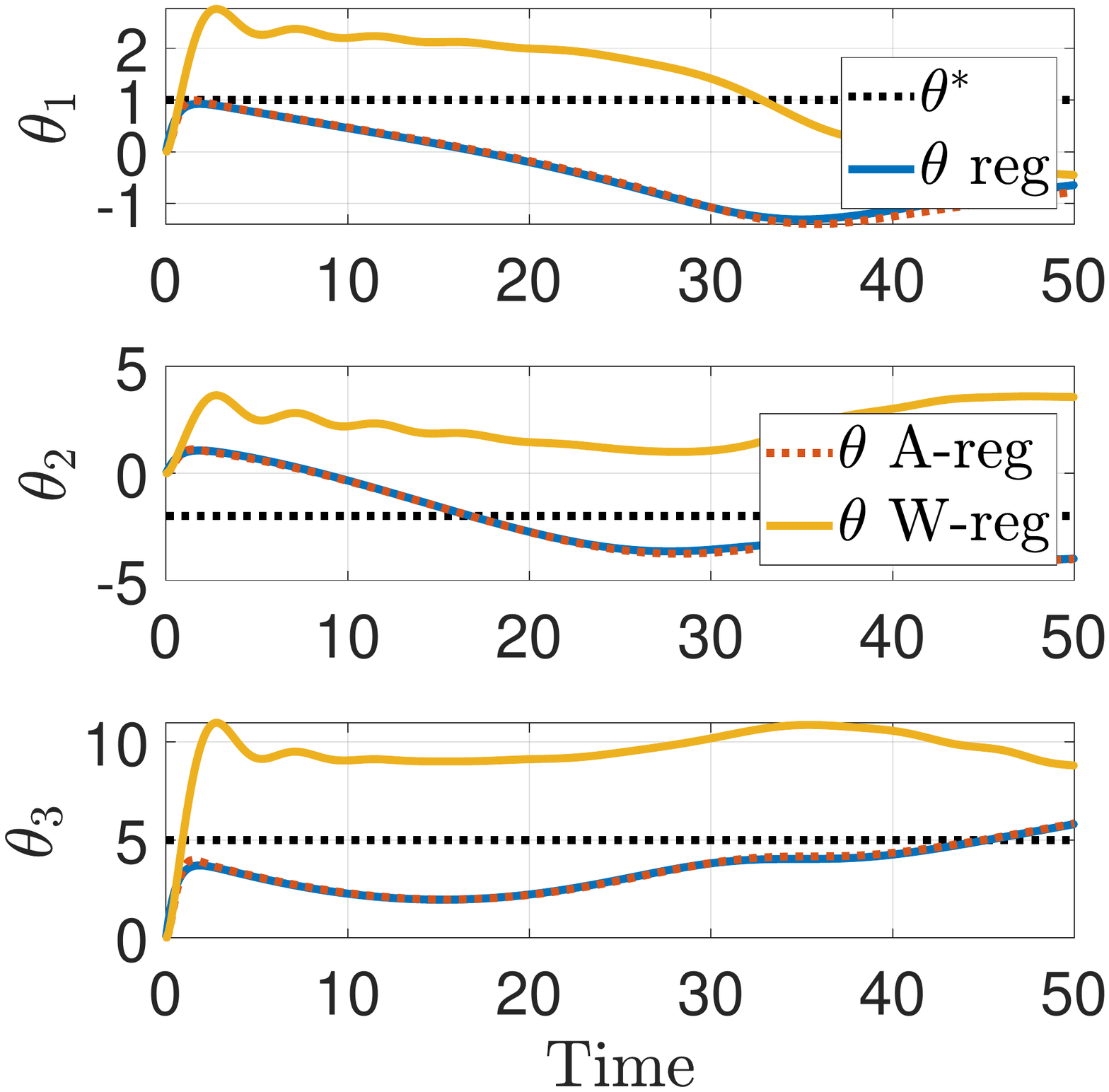}
	    \includegraphics[trim={0.25cm 3.5cm 1.5cm 4.4cm},clip,width=0.2\textwidth]{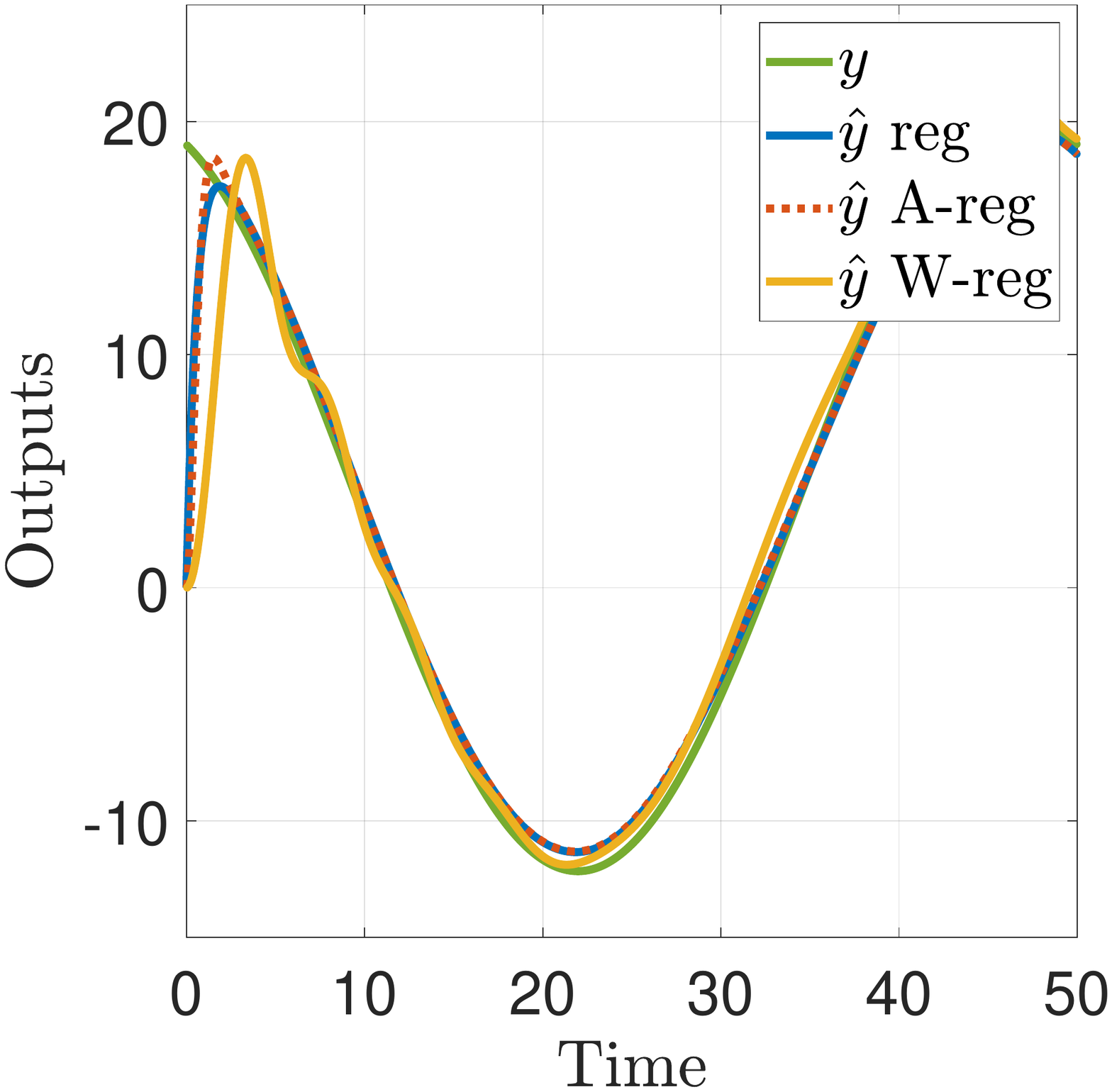}
	    \includegraphics[trim={0.25cm 3.5cm 1.5cm 4.4cm},clip,width=0.2\textwidth]{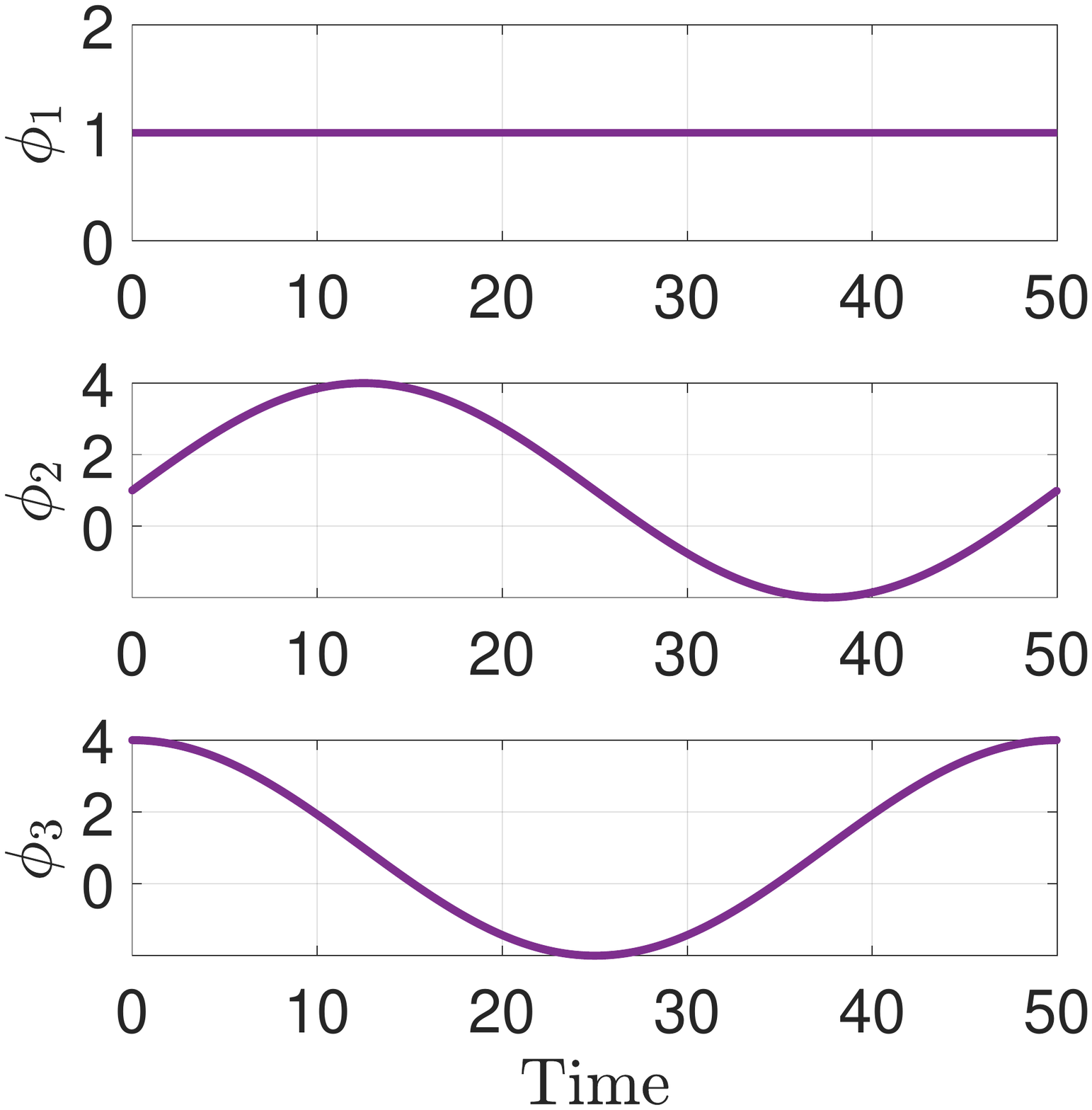}
	    }
	\caption{Time-varying regression: Response with $\phi=[1,~1+3\sin(\frac{2\pi}{50}t),~1+3\cos(\frac{2\pi}{50}t)]^T$.}
	\label{f:Error1_One_Period_Appendix}
    \end{subfigure}
    
    \begin{subfigure}[b]{\textwidth}
        \centerline{
	    \includegraphics[trim={0.25cm 3.5cm 1.5cm 4.4cm},clip,width=0.2\textwidth]{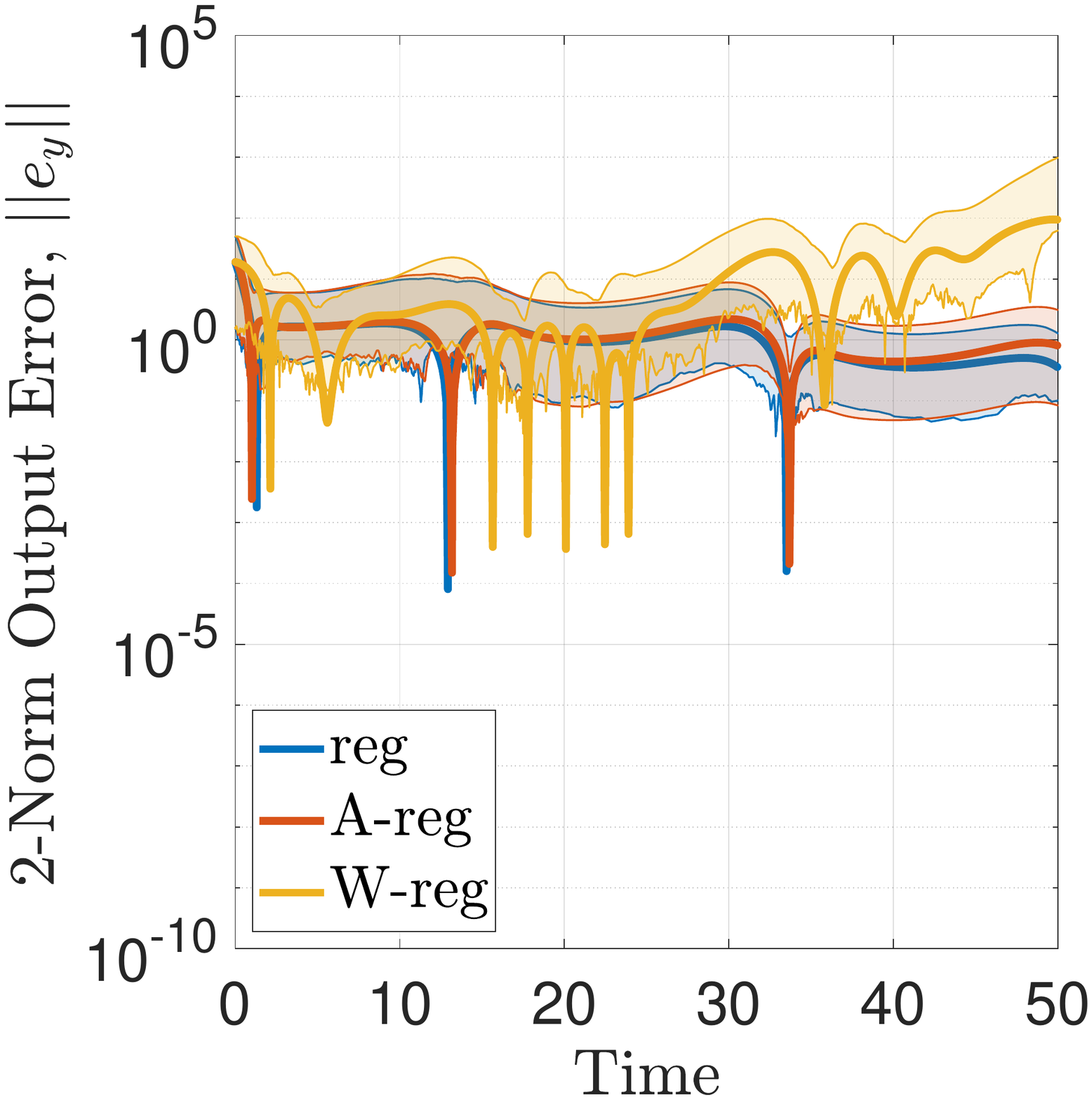}
	    \includegraphics[trim={0.25cm 3.5cm 1.5cm 4.4cm},clip,width=0.2\textwidth]{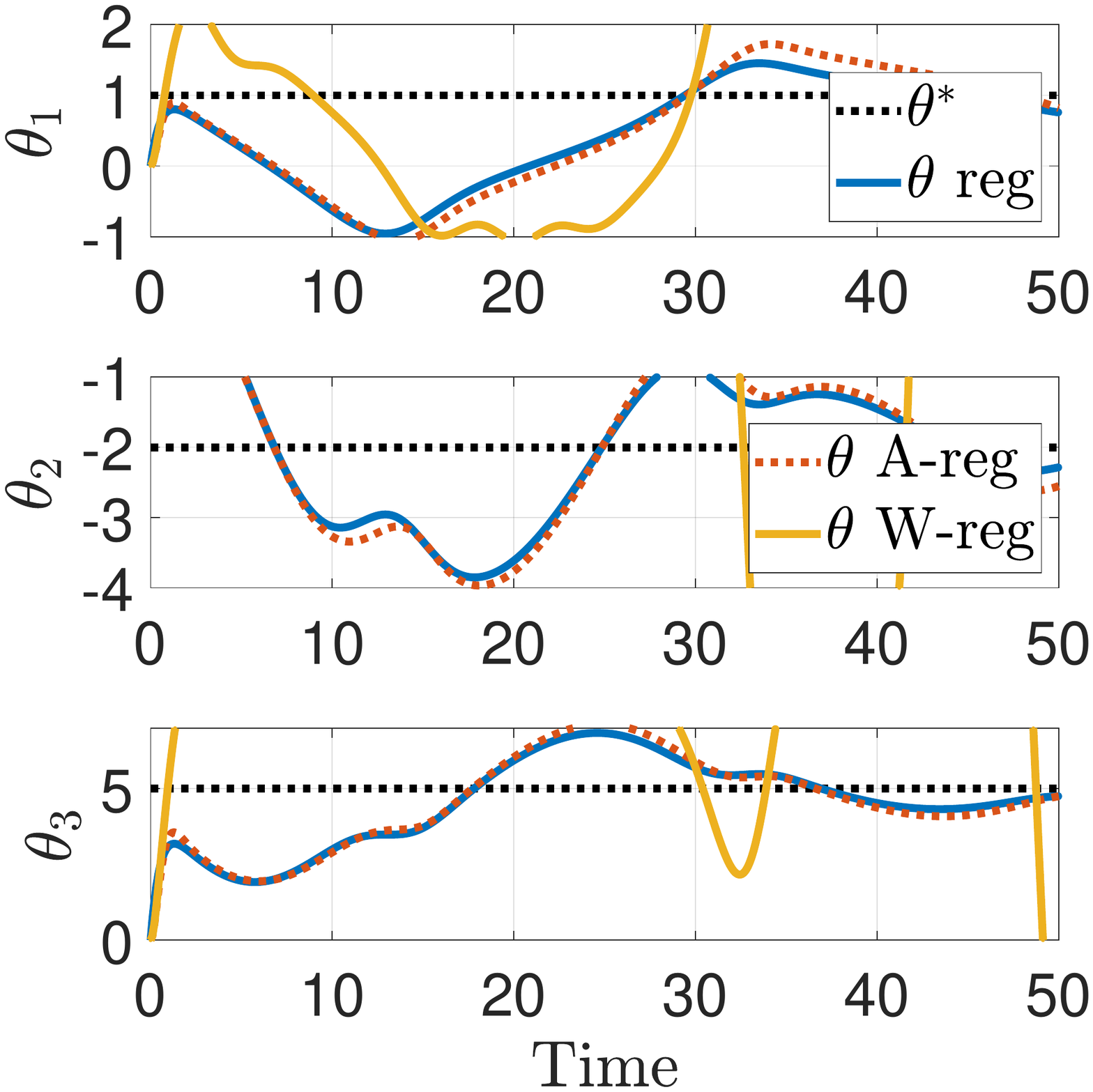}
	    \includegraphics[trim={0.25cm 3.5cm 1.5cm 4.4cm},clip,width=0.2\textwidth]{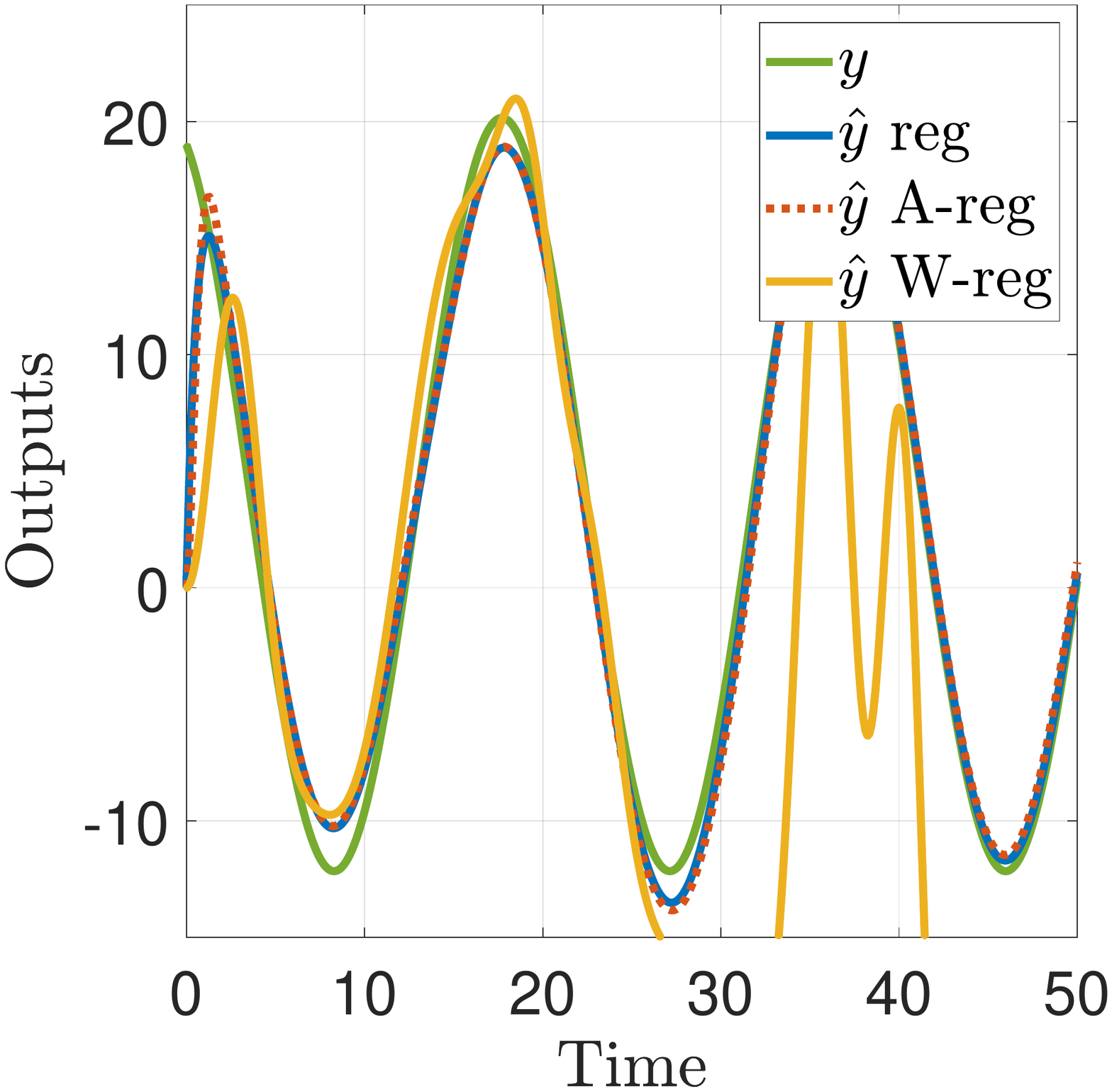}
	    \includegraphics[trim={0.25cm 3.5cm 1.5cm 4.4cm},clip,width=0.2\textwidth]{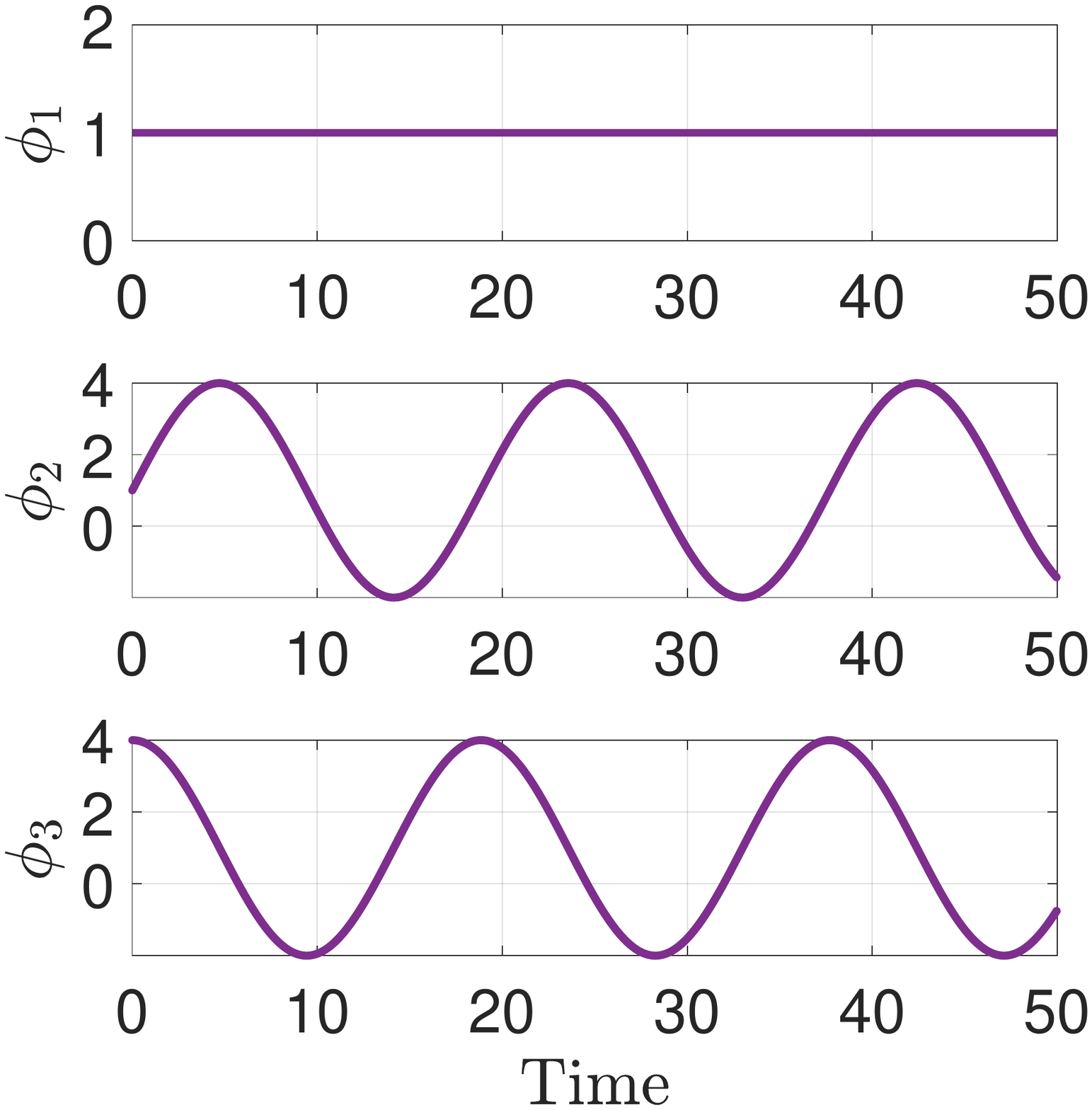}
	    }
	\caption{Time-varying regression: Response with $\phi=[1,~1+3\sin(\frac{1}{3}t),~1+3\cos(\frac{1}{3}t)]^T$.}
	\label{f:Error1_Slow_Response_Appendix}
    \end{subfigure}
    
    \begin{subfigure}[b]{\textwidth}
        \centerline{
	    \includegraphics[trim={0.25cm 3.5cm 1.5cm 4.4cm},clip,width=0.2\textwidth]{norm_error_y_1_log_bars_PE.pdf}
	    \includegraphics[trim={0.25cm 3.5cm 1.5cm 4.4cm},clip,width=0.2\textwidth]{theta_y_PE.pdf}
	    \includegraphics[trim={0.25cm 3.5cm 1.5cm 4.4cm},clip,width=0.2\textwidth]{y_PE.pdf}
	    \includegraphics[trim={0.25cm 3.5cm 1.5cm 4.4cm},clip,width=0.2\textwidth]{phi_PE.pdf}
	    }
	\caption{Time-varying regression: Response with $\phi=[1,~1+3\sin(t),~1+3\cos(t)]^T$.}
	\label{f:Error1_PE_Response_Appendix}
    \end{subfigure}
    
    \begin{subfigure}[b]{\textwidth}
        \centerline{
	    \includegraphics[trim={0.25cm 3.5cm 1.5cm 4.4cm},clip,width=0.2\textwidth]{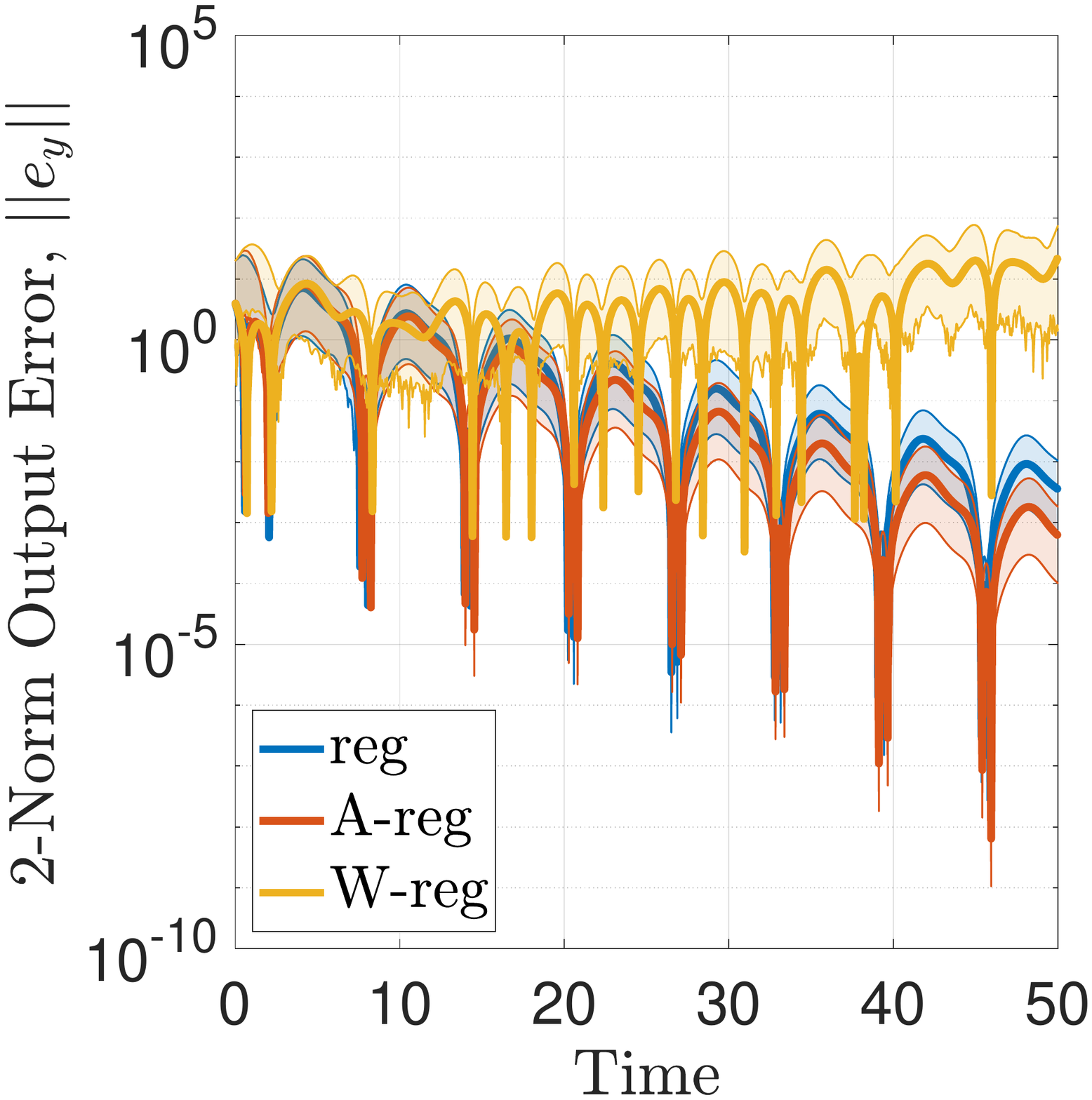}
	    \includegraphics[trim={0.25cm 3.5cm 1.5cm 4.4cm},clip,width=0.2\textwidth]{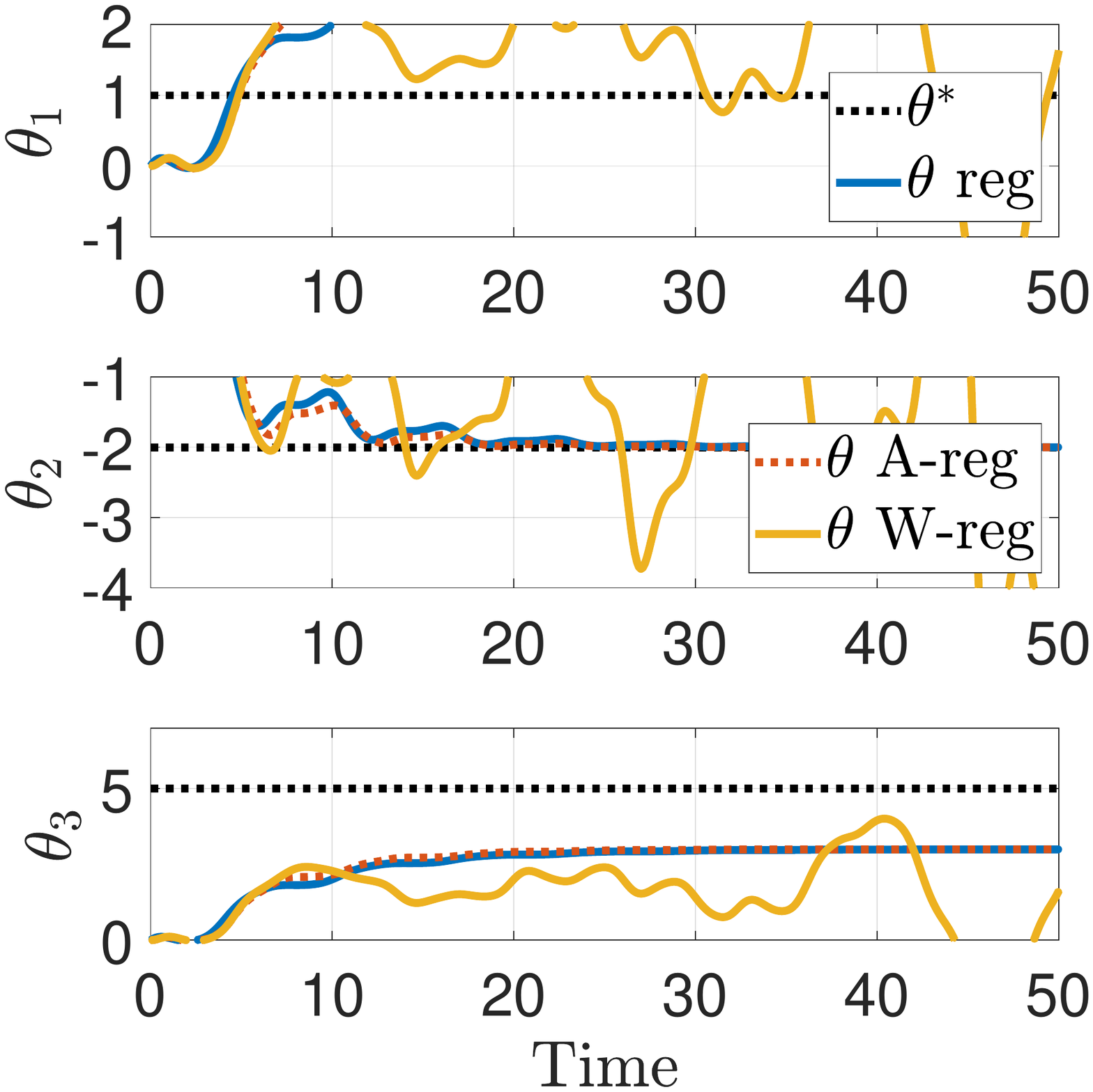}
	    \includegraphics[trim={0.25cm 3.5cm 1.5cm 4.4cm},clip,width=0.2\textwidth]{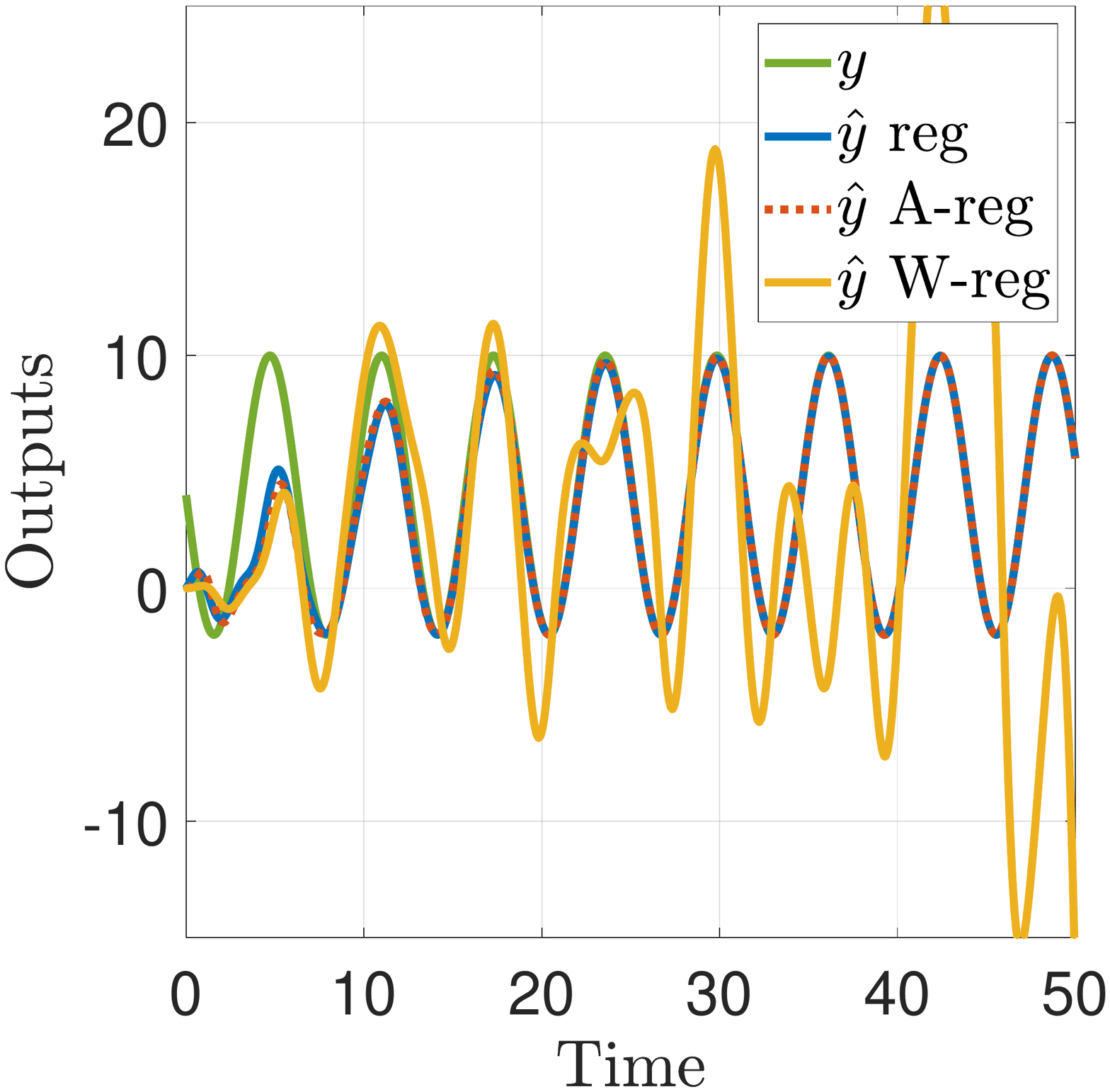}
	    \includegraphics[trim={0.25cm 3.5cm 1.5cm 4.4cm},clip,width=0.2\textwidth]{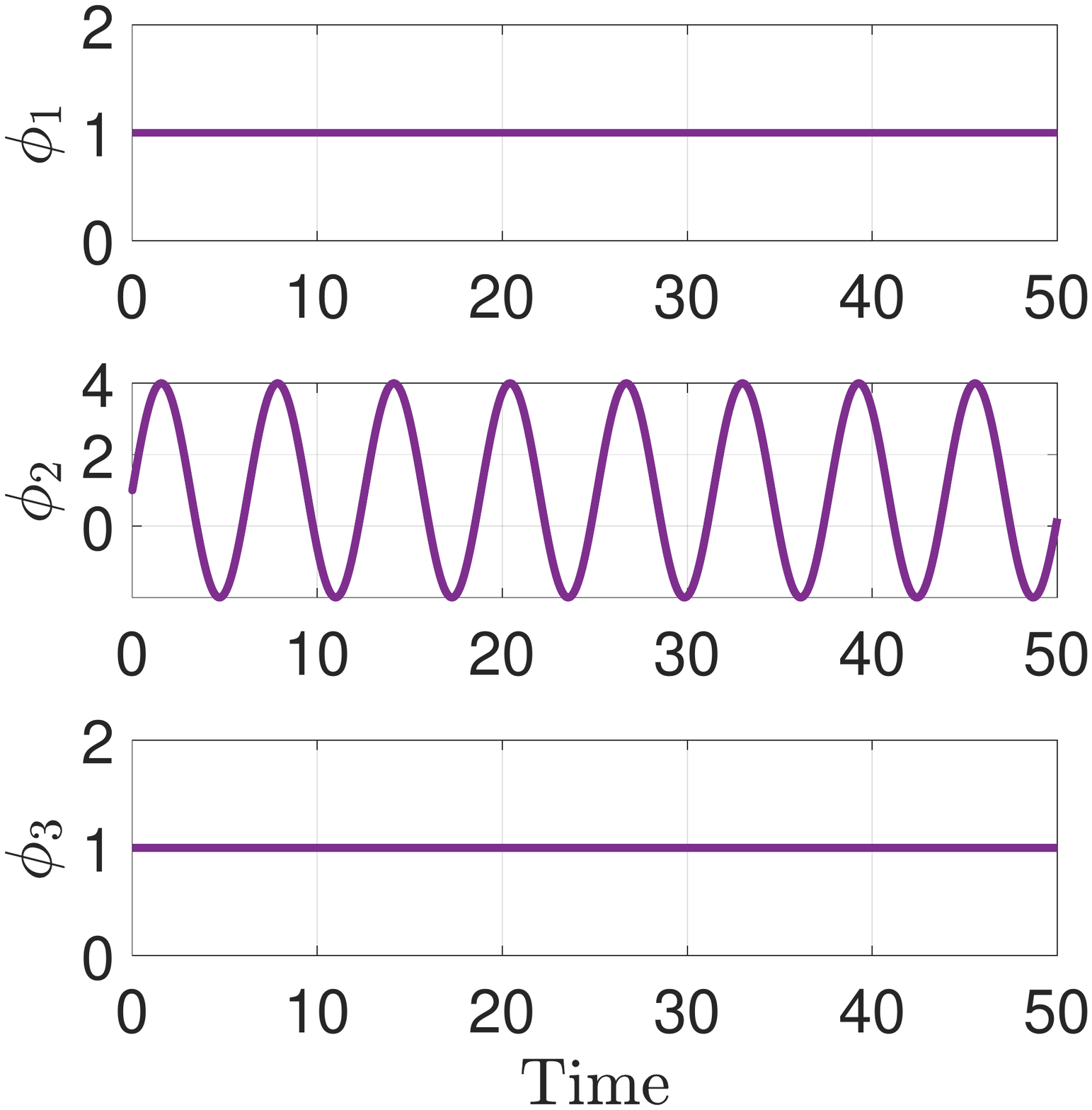}
	    }
	\caption{Time-varying regression: Response with $\phi=[1,~1+3\sin(t),~1]^T$.}
	\label{f:Error1_1sin_Response_Appendix}
    \end{subfigure}
    \caption{(to be viewed in color) Left plot: Output error trajectories. Left-middle: Parameter trajectories. Right-middle: Output trajectories. Right: Time-varying feature trajectories. $95\%$ intervals for error plots shown as shaded regions. Example trajectories shown as solid and dashed lines. This figure shows a progression in the increase of the frequency of the features of Section \ref{ss:TV_Regression_Experiments}.}
    \label{f:figure_collection2}
\end{figure}

\end{document}